\documentclass[a4paper,11pt,reqno]{amsart} 
\usepackage{fourier} 
\usepackage{amssymb}
\usepackage{amsmath}
\usepackage{amsthm}

\usepackage{hyperref}
\usepackage{fullpage}
\usepackage{xcolor} 
\usepackage{listings}
\theoremstyle{plain}
\newtheorem{thm}{Theorem}[section]

\newtheorem{cor}[thm]{Corollary}
\theoremstyle{definition}

\newtheorem*{notn*}{Notation}

\newtheorem{remark}[thm]{Remark}

\def\Z{\mathbb{Z}}
\def\C{\mathbb{C}}

\def\V{\mathbb{V}}
\def\fS{\mathfrak{S}}
\def\AA{\mathcal{A}}

\newcommand{\Sym}{{\mathrm{Sym}}}
\newcommand{\Sp}{{\mathrm{Sp}}}
\newcommand{\SL}{{\mathrm{SL}}}
\newcommand{\GL}{{\mathrm{GL}}}

\newcommand{\iso}{{\mathrm{iso}}}

\newenvironment{Tabular}[2][1]
  {\def\arraystretch{#1}\tabular{#2}}
  {\endtabular}

\title{Dimension formulas for spaces of vector-valued Siegel modular forms of degree two and level two}
\author{Jonas Bergstr\"om and Fabien Cl\'ery }
\date{}

\email{jonasb@math.su.se}
\email{cleryfabien@gmail.com}

\begin{document}

\begin{abstract}
Using a description of the cohomology of local systems on the moduli space of abelian surfaces with a full level two structure, together with a computation of Euler characteristics we find the isotypical decomposition, under the symmetric group on 6 letters, of spaces of vector-valued Siegel modular forms of degree two and level two.
\end{abstract}

\maketitle

\section{Introduction}
In this paper, we refine the previously known dimension formulas for spaces 
$M_{k,j}(\Gamma[2])$ of vector-valued Siegel modular forms of degree 2 and level 2, 
by determining their isotypical decomposition under the action of $\Sp(4,\Z/2\Z)\cong \mathfrak{S}_6$.
This extends previous work, see for instance 
\cite{Igusa64, Arakawa, Tsushima1982, Tsushima, Ibukiyama84, Ibukiyama07, Wakatsuki,  CvdGG}. 
In particular, Tsushima gave in \cite[Theorems 2, 3]{Tsushima} a formula for the dimension of the space 
$S_{k,j}(\Gamma[N])$ for any $N$ under the conditions $j\geqslant 1$ and $k\geqslant 5$ or $j=0$ and $k\geqslant 4$. The 
ranges for $j$ and $k$ in Tsushima's dimension formula for $N=2$ have been slightly extended in 
\cite[Theorem 12.1]{CvdGG}. There is an overview of dimension formulas such as these on the webpage \cite{Schmidt}. 

In \cite{BFvdG}, there is a conjectural description of the motivic Euler characteristic, with its isotypical decomposition under the action of $\mathfrak{S}_6$, of any symplectic local system on $\mathcal A_2[2]$ the moduli space of abelian surfaces with a full level two structure. These conjectures were later proven in \cite{Roesner}. In particular, this gives us the integer-valued Euler characteristic of an isotypical component under $\mathfrak{S}_6$ of any local system on $\mathcal A_2[2]$ as a sum of a well-known value (in terms of dimensions of spaces of elliptic modular cusp forms) plus four times the dimension of the (isotypical component of the) vector space of Siegel modular cusp forms of degree 2 and level 2, see Theorem~\ref{thm-main}. In Section~\ref{sec-euler} we then find an effective formula to compute these integer-valued Euler characteristics. This is achieved by stratifying the moduli space $\mathcal A_2[2]$ in terms of the automorphism groups of principally polarized abelian surfaces, which are Jacobians of smooth projective curves of genus $2$, or products of elliptic curves. By computing the action of these automorphism groups and of $\mathfrak{S}_6$, on the first cohomology group of the corresponding abelian surfaces, we can find a formula for the integer-valued Euler characteristic, see Equation~\eqref{eq-numeric}. This is a method previously used for instance in \cite{Getzler}. 

In Section~\ref{sec-smf} we give an overview of the Siegel modular forms we are interested in, together with a short description of the Arthur packets for $\mathrm{GSp}(4)$. Then, in Section~\ref{sec-sv} and Section~\ref{sec-vv} we include isotypical decompositions of the spaces of Siegel modular forms of degree $2$ and level $2$ to give a comprehensive reference for these results. 
Computer programs, written in Sage, which compute all results of this paper, are provided on a GitHub repository \cite{BC}. Tables with some of the results of this paper can also be found on the webpage \cite{BCFvdGweb}.

\section{Siegel modular forms} \label{sec-smf}
The level 2 congruence subgroups we are concerned with are  the following ones
\[
\Gamma[2]
=
\left \{
\gamma \in \Gamma : \gamma \equiv 1_4 \bmod 2
\right \},
\quad
\Gamma_1[2]
=
\left \{
\gamma \in \Gamma: 
\gamma 
\equiv 
\left(\begin{smallmatrix}1_2 & * \\ 0 & 1_2\end{smallmatrix}\right) 
\bmod 2
\right \}, 
\quad
\Gamma_0[2]
=
\left \{
\left(
\begin{smallmatrix}
a & b \\ c & d
\end{smallmatrix}
\right) \in \Gamma: c \equiv 0 \bmod 2
\right \},
\]
where $\Gamma=\Sp(4,\Z)=\{\gamma \in \GL(4,\Z): \gamma^t J \gamma =J\}$ with 
$J=\left(\begin{smallmatrix}0 & 1_2 \\ -1_2 & 0\end{smallmatrix}\right)$ and $1_n$ the identity matrix of size $n$. 
We clearly have the following inclusions 
$\Gamma[2] < \Gamma_1[2] < \Gamma_0[2] < \Gamma$ and the successive quotients can be identified as follows (see \cite[Sections 2, 3, 9]{CvdGG})
\begin{equation}\label{IdenQuotient}
\Gamma_1[2]/\Gamma[2] \cong (\Z/2\Z)^3, \quad
\Gamma_0[2]/\Gamma[2] \cong \fS_4 \times  \Z/2\Z, \quad
\Gamma_0[2]/\Gamma_1[2] \cong \fS_3, \quad
\Gamma/\Gamma[2] \cong \fS_6
\end{equation}
with $\mathfrak{S}_n$ the symmetric group on $n$ letters. As usual, theses groups act on the Siegel upper half space $\mathfrak{H}_2$ of degree 2:
$\mathfrak{H}_2=\{\tau \in \text{Mat}(2\times 2,\C) : \tau^t=\tau, \text{Im}(\tau)>0 \}$ via $\tau \mapsto \gamma\,\tau=(a\tau+b)(c\tau+d)^{-1}$.
More details on the orbifolds of the action of the previous groups can be
found in \cite[Section 2]{CvdGG}. 
For any integer $k$, any non-negative integer $j$ and any 
$\gamma=\left(
\begin{smallmatrix}
a & b \\ c & d
\end{smallmatrix}
\right)\in \Gamma$, 
we define a slash operator on functions 
$f\colon \mathfrak{H}_2 \to \text{Sym}^j(\C^2)$,
\[
\big(f\vert_{k,j} \gamma\big) (\tau)=
\det(c\tau +d)^{-k}
\otimes
\text{Sym}^j((c\tau +d)^{-1})\:f((a\tau+b)(c\tau +d)^{-1}).
\]
We let $G$ be one of the groups $\Gamma[2]$, $\Gamma_1[2]$, $\Gamma_0[2]$ or $\Gamma$.
The space of modular forms of weight $(k,j)$ on $G$ is denoted by $M_{k,j}(G)$ and is defined by
\[
M_{k,j}(G)=
\left\{
f \colon \mathfrak{H}_2 \to \text{Sym}^j(\C^2)\, | \,  f \, \text{holomorphic}, \,
f\vert_{k,j} \gamma=f
\text{ for any }\, 
\gamma\in G
\right\}.
\]
The subspace of cusp forms of $M_{k,j}(G)$ will be denoted by $S_{k,j}(G)$, this is the kernel of the (global) Siegel $\Phi$-operator.
Let us make a couple of easily verified remarks. Firstly, since $-1_4$ belongs to the group $\Gamma[2]$, 
we have that $M_{k,j}(\Gamma[2])=\{0\}$ if $j$ is odd. This can be directly read off from the functional equation satisfied by any 
element of $M_{k,j}(\Gamma[2])$. Therefore, from now on, we assume that 
{\textit{$j$ is even}}. 
Secondly, if $k$ is odd then $M_{k,j}(\Gamma[2])=S_{k,j}(\Gamma[2])$: 
let $\Gamma(2)$ be the principal congruence subgroup 
\footnote
{
We denote congruence subgroups of $\SL(2,\Z)$ by round brackets and those of $\Sp(4,\Z)$ by square brackets.
} 
of level $2$ of $\SL(2,\Z)$, the (global) Siegel $\Phi$-operator maps $M_{k,j}(\Gamma[2])$ to $M_{j+k}(\Gamma(2))^{\oplus 15}$ 
(note $15$ is the number of $1$-dimensional cusps of the group $\Gamma[2]$) and since $-1_2$ 
belongs to $\Gamma(2)$, the space $M_{j+k}(\Gamma(2))$ reduces to ${0}$ when $k$ is odd ($j$ is even). These 
two facts also hold for the groups $\Gamma_1[2]$, $\Gamma_0[2]$ and $\Gamma$.  
A less easy fact is the generalisation of the Koecher 
principle to vector-valued Siegel modular forms, see \cite[Satz 1]{Freitag79}, which implies that 
\[
M_{k,j}(\Gamma[2])=\left\{0\right\} \text{ for any } k<0 \text{ and any } j.
\]

The Petersson inner product provides an orthogonal decomposition 
\begin{equation}\label{OrthoDecom}
M_{k,j}(G)=E_{k,j}(G)\oplus S_{k,j}(G).
\end{equation}

We call $E_{k,j}(G)$ the space of Eisenstein series.
The space of Eisenstein series can in turn be decomposed into $2$ pieces, orthogonal with respect to the Petersson inner product, as follows (see \cite[Proposition 13.2]{CvdGG})
\[
E_{k,j}(G)=SE_{k,j}(G)\oplus KE_{k,j}(G)
\]
where $SE_{k,j}(G)$ denotes the space of Siegel-Eisenstein series (they map to constants under the Siegel $\Phi$-operator) and 
$KE_{k,j}(G)$ denotes the space of Klingen-Eisenstein series (they map to cusp forms of degree $1$ under the Siegel $\Phi$-operator).
The decomposition of the space of cusp forms $S_{k,j}(G)$ can also be refined according to the classification of automorphic representations of G$\Sp(4)$: Arthur packets. There are six different types of Arthur packets 
(see \cite{Arthur}) but only three of them can appear in our situation (see \cite[Section 2.1]{Schmidt18} and \cite[Proposition 4.3]{Roy_Schmidt_Yi}), namely:
\begin{itemize}
\item $\textbf{(G)}$: general type.  
\item $\textbf{(Y)}$: Yoshida type. They can only appear in $S_{k,j}(G)$ with $j>0$, i.e., in the case of vector-valued cusp forms. 
Modular forms of this type are also called \textbf{Yoshida lifts}.
\item $\textbf{(P)}$: Saito-Kurokawa type (Siegel parabolic). 
They can only appear in $S_{k,0}(G)$, i.e., in the case of scalar-valued cusp forms. 
Modular forms of this type are also called \textbf{Saito-Kurokawa lifts}.
\end{itemize}
The other three types of Arthur packets (Howe--Piatetski-Shapiro type (Borel parabolic), finite type and Soudry type (Klingen parabolic) do not appear in our situation.
So, in summary, we have 
\begin{equation}\label{DecompPackets}
E_{k,j}(G)=SE_{k,j}(G)\oplus KE_{k,j}(G) \quad
\text{and} \quad
S_{k,j}(G)=
S^{\textbf{(G)}}_{k,j}(G)\oplus S^{\textbf{(P)}}_{k,j}(G)\oplus 
S^{\textbf{(Y)}}_{k,j}(G).
\end{equation}
Note that the last decomposition is also orthogonal with respect to the Petersson inner product.
Since the group $\Gamma[2]$ is a normal subgroup of  $\Gamma$ (it is the kernel of the reduction modulo $2$), 
we get an action of $\Gamma$ on the space $M_{k,j}(\Gamma[2])$
\[
\begin{smallmatrix}
\Sp(4,\Z) \times M_{k,j}(\Gamma[2]) & \to & M_{k,j}(\Gamma[2])\\
(\gamma,f) & \mapsto & f\vert_{k,j} \gamma^{-1}
\end{smallmatrix}
\]
From this action, we deduce a group homomorphism
\[
\begin{smallmatrix}
\Gamma & \to & \GL(M_{k,j}(\Gamma[2]))\\
\gamma & \mapsto &
\left(
\begin{smallmatrix}
M_{k,j}(\Gamma[2]) & 
\to & M_{k,j}(\Gamma[2])\\
f & \mapsto & f\vert_{k,j} \gamma^{-1}
\end{smallmatrix}
\right)
\end{smallmatrix}
\]
whose kernel obviously contains the group $\Gamma[2]$. So the previous homomorphism factors through the group $\Gamma[2]$ and we obtain a group homomorphism 

\[
\Gamma/\Gamma[2]  \cong \Sp(4,\Z/2\Z) \cong \mathfrak{S}_6  \to  \GL(M_{k,j}(\Gamma[2]))
\]
i.e., a representation of the group $\mathfrak{S}_6$ on the space $M_{k,j}(\Gamma[2])$. 
Note that the second isomorphism is ambiguous due to the outer automorphism of $\mathfrak{S}_6$ so 
we need to fix this isomorphism. 
We fix this isomorphism as follows:
$\mathfrak{S}_6=\langle(12),(123456)\rangle$ and as in \cite[ Equation (3.2)]{CvdGG} (see also \cite[pp. 398-399]{Igusa64}), we set
\[
(12) \mapsto 
\left(
\begin{smallmatrix}
1 & 0 & 1 & 0 \\
0 & 1 & 0 & 0 \\
0 & 0 & 1 & 0 \\
0 & 0 & 0 & 1 
\end{smallmatrix}
\right) \bmod 2
\quad
\text{and}
\quad
(123456) \mapsto 
\left(
\begin{smallmatrix}
0 & 1 & 0 & 1 \\
1 & 0 & 1 & 0 \\
1 & 0 & 1 & 1 \\
-1 & 1 & 0 & 1 
\end{smallmatrix}
\right) \bmod 2.
\]

The irreducible representations of $\mathfrak{S}_n$ correspond bijectively with the partitions of $n$.
The representation of the symmetric group $\mathfrak{S}_n$ corresponding with the partition $\varpi$ will be denoted by $s[\varpi]$, 
with $s[n]$ the trivial one and $s[1^n]$ 
the alternating one. 
\begin{table*}[ht]\label{Irrep_S_6}
\caption{Irreducible representations of $\fS_6$ and their dimensions}
\begin{center}
\begin{tabular}{c|c|c|c|c|c|c|c|c|c|c|c}
    $[\varpi]$ & $[6]$ & $[5,1]$ & $[4,2]$ & $[4,1^2]$ & $[3^2]$ & $[3,2,1]$ & $[3,1^3]$ & $[2^3]$ & $[2^2,1^2]$ & $[2,1^4]$ & $[1^6]$ \\
    \hline
    $\dim s[\varpi]$ & 1 & 5 & 9 & 10 & 5 & 16 & 10 & 5 & 9 & 5 & 1 
\end{tabular}
\label{table:IrrepList}
\end{center}
\end{table*}

\begin{remark}\label{Character}
Let $M_{k,j}(\Gamma,\epsilon)$ be the space of modular forms of weight $(k,j)$ on $\Gamma$ with character $\epsilon$ where $\epsilon$ is the unique non-trivial character of $\Gamma$, see \cite[Section 12]{CFvdG19} for a brief description of this character. 
Then we have $M_{k,j}(\Gamma,\epsilon)=M_{k,j}(\Gamma[2])^{s[1^6]}$ where $M_{k,j}(\Gamma[2])^{s[1^6]}$ denotes the $\fS_6$-anti-invariant subspace of $M_{k,j}(\Gamma[2])$, i.e., 
$M_{k,j}(\Gamma[2])^{s[1^6]}=
\left\{
f \in M_{k,j}(\Gamma[2]) \, | \, f\vert_{k,j} \sigma =\text{sgn}(\sigma) f 
\text{ for any }\, 
\sigma \in \fS_6
\right\}$
where $\text{sgn}(\sigma)$ denotes the signature of $\sigma$. 
\end{remark}
\begin{remark}\label{S3_rep_space}
As the congruence subgroup $\Gamma_1[2]$ is a normal subgroup of $\Gamma_0[2]$, in the same way we turned the space $M_{k,j}(\Gamma[2])$ into a $\fS_6$-representation space, we can turn the space 
$M_{k,j}(\Gamma_1[2])$ into a $\fS_3$-representation space (see last but one isomorphism in (\ref{IdenQuotient}) and \cite[Eq.(3.3)]{CvdGG} for an explicit description of this quotient). 
\end{remark}

For a $\fS_n$-representation space $V$, we write its isotypical decomposition as 
\begin{equation}\label{iso_Sn}
\iso_{\fS_n} V=\sum_{\varpi \vdash n} m_{s[\varpi]}(V) \cdot s[\varpi] \in \Z[\fS_n],
\end{equation}
where $\Z[\fS_n]$ is the representation ring and $m_{s[\varpi]}(V)$ is the multiplicity of the representation $s[\varpi]$ appearing in $V$. We call the right hand side of (\ref{iso_Sn}) the isotypical decomposition of $V$. 

Knowing the isotypical decomposition of a space $M_{k,j}(\Gamma[2])$ gives us all the information we want about the spaces $M_{k,j}(\Gamma_1[2])$, $M_{k,j}(\Gamma_0[2])$ and 
$M_{k,j}(\Gamma)$ by representation theory.
The last isomorphism in (\ref{IdenQuotient}) tells us that the space $M_{k,j}(\Gamma)$ is  the invariant subspace of $M_{k,j}(\Gamma[2])$ under the action of the symmetric group $\fS_6$.
From the first isomorphism in (\ref{IdenQuotient}), we see that  
$M_{k,j}(\Gamma_1[2])$ is the invariant subspace of $M_{k,j}(\Gamma[2])$ under the action of the group $(\Z/2\Z)^3\cong \langle (12), (34), (56)\rangle$. 
Write 
\[
\iso_{\mathfrak{S_6}}M_{k,j}(\Gamma[2])=
m_{s[6]}\,s[6]+m_{s[5,1]}\,s[5,1]+\cdots+m_{s[1^6]}\,s[1^6] 
\]
then we have
\begin{align*}
&\iso_{\mathfrak{S_3}}M_{k,j}(\Gamma_1[2]) = 
(m_{s[6]}+m_{s[4,2]}+m_{s[2^3]})s[3]+
(m_{s[5,1]}+m_{s[4,2]}+m_{s[3,2,1]})s[2,1]+ 
(m_{s[4,1^2]}+m_{s[3^2]})s[1^3], \\ 
&\dim M_{k,j}(\Gamma_0[2])=
M_{k,j}(\Gamma_1[2])^{\fS_3}=m_{s[6]}+m_{s[4,2]}+m_{s[2^3]},\\
&\dim M_{k,j}(\Gamma)=\dim M_{k,j}(\Gamma[2])^{\fS_6}=m_{s[6]},\\
&\dim M_{k,j}(\Gamma,\varepsilon)=\dim M_{k,j}(\Gamma[2])^{s[1^6]}=m_{s[1^6]}
\end{align*}
where the last formula comes from Remark \ref{Character} and the first two come from \cite[Sections 3 and 9]{CvdGG}.
Therefore we focus on the spaces  $M_{k,j}(\Gamma[2])$ in the sequel. 
In the case of scalar-valued modular forms, the previous decompositions allow us to recover the results given in \cite[Appendix A]{Roy_Schmidt_Yi} for the groups $\Gamma, \Gamma_0[2]$ and $\Gamma[2]$. 
In the case of scalar-valued modular forms we simply write
\[
M_{k}(\Gamma[2])=M_{k,0}(\Gamma[2]) \quad \text{and} \quad S_{k}(\Gamma[2])=S_{k,0}(\Gamma[2]).
\]

\section{Isotypical decompositions in the scalar-valued case} \label{sec-sv}
By the Koecher principle, we know that 
$M_{k}(\Gamma[2])=\{0\}$ for $k<0$.
The following theorem is due to Igusa, see \cite[p.398]{Igusa64}.
\begin{thm}[Igusa]\label{Igusa}
We have
\[
\dim M_{k}(\Gamma[2])=
\begin{cases}
\frac{(k+1)(k^2+2k+12)}{12} \hspace{20pt}   \text{if}\,\,\, k\geqslant 0 \,\,\, even\\
\dim M_{k-5}(\Gamma[2]) \hspace{13pt} \text{if} \,\,\,\, k\geqslant 1\,\,\, \text{odd}.
\end{cases}
\]
\end{thm}

For $k$ odd, the last equality comes from 
$M_{k}(\Gamma[2])=S_{k}(\Gamma[2])=\chi_5 \cdot M_{k-5}(\Gamma[2])$ 
where $\chi_5$ denotes the unique cusp form, up to a multiplicative constant, generating the space $S_{5}(\Gamma[2])$. 
In fact, Igusa did more than computing the dimension of the spaces $M_{k}(\Gamma[2])$.
He also computed the characters of $\mathfrak{S}_6$ on the space $M_{k}(\Gamma[2])$
(see  \cite[Theorem 2]{Igusa64}) and he showed that 
as $\mathfrak{S}_6$-representation for $k$ even  we have
\[
\iso_{\fS_6} M_{k}(\Gamma[2])={\text{Sym}}^{k/2}(s[2^3])-
\begin{cases}
0 \hspace{2.75cm}\text{ if } k \in \{0,2,4,6\}\\
{\text{Sym}}^{k/2-4}(s[2^3])  \,\,\,\,\,\,\,\,\,\,\,\, \text{ if } k \geqslant 8, 
\end{cases}
\]
where we put 
$\Sym^0(s[\varpi])=s[n]$ for any irreducible representation $s[\varpi]$ of $\mathfrak{S}_n$.
Note that the relation appearing in weight $8$ defines the Igusa quartic. 
\begin{thm}[Igusa]\label{Igusa2}
The following table gives the generating series for the multiplicity of 
the irreducible representations of $\,\fS_6$ in $M_{k}(\Gamma[2])$

\begin{center}
\begin{Tabular}[1.3]{c|c}
$s[\varpi]$ & $\sum_{k \geqslant 0}m_{{s[\varpi]}}(M_k(\Gamma[2]))\,t^k$ \\
\hline
$s[6]$ & $\frac{1+t^{35}}{(1-t^4)(1-t^6)(1-t^{10})(1-t^{12})}$   \\
\hline
$s[5,1]$ & $\frac{t^{11}(1+t)}{((1-t^4)(1-t^6))^2}$   \\
\hline
$s[4,2]$ & $\frac{t^4(1+t^{15})}{(1-t^2)(1-t^4)^2(1-t^{10})}$   \\
\hline
$s[4,1^2]$ & $\frac{t^{11}(1+t^4)}{(1-t)(1-t^4)(1-t^6)(1-t^{12})}$   \\
\hline
$s[3^2]$  & $\frac{t^7(1+t^{13})}{(1-t^2)(1-t^4)(1-t^6)(1-t^{12})}$   \\
\hline
$s[3,2,1]$ & $\frac{t^8(1-t^8)}{(1-t^2)^2(1-t^5)(1-t^6)^2}$   \\
\hline
$s[3,1^3]$ & $\frac{t^6(1+t^4+t^{11}+t^{15})}{(1-t^2)(1-t^4)(1-t^6)(1-t^{12})}$  \\
\hline
$s[2^3]$ & $\frac{t^2(1+t^{23})}{(1-t^2)(1-t^4)(1-t^6)(1-t^{12})}$   \\
\hline
$s[2^2,1^2]$ & $\frac{t^9}{(1-t^2)(1-t^4)^2(1-t^5)}$  \\
\hline
$s[2,1^4]$ & $\frac{t^6(1+t^{11})}{((1-t^4)(1-t^6))^2}$ \\
\hline
$s[1^6]$ & $\frac{t^5(1+t^{25})}{(1-t^4)(1-t^6)(1-t^{10})(1-t^{12})}$  \\
\end{Tabular}
\end{center}
\end{thm}
\begin{proof}
Let $X_k$ be the character of the representation of $\fS_6$ on the space 
$M_{k}(\Gamma[2])$ and let us denote by $\rho$ the irreducible representation $s[2^3]$. Then, for any $\sigma\in \fS_6$, Theorem 2 of \cite{Igusa64} gives
\[
\sum_{k\geqslant 0} X_k(\sigma)t^k=(1+\text{sgn}(\sigma)t^5)(1-t^8)/\det(1_5-\rho(\sigma)t^2).
\]
Clearly this formula depends only on the conjugacy class of $\sigma$ and Igusa, see \cite[p. 401]{Igusa64},
gave the expression of $f_{\sigma}(t)=\det(1_5-\rho(\sigma)t)$ for each irreducible representation of $\fS_6$. From the formula (see  for example \cite[p. 401]{Igusa64}) giving the multiplicity of an irreducible representation in a representation space of a finite group, we deduce that the multiplicity of the irreducible representation $s[\varpi]$ in $M_{k}(\Gamma[2])$ is given by 
\[
m_{{s[\varpi]}}(M_k(\Gamma[2])=
\frac1{|\fS_6|}\, \sum_{\sigma\in \fS_6}X_k(\sigma)\,X_{s[\varpi]}(\sigma^{-1})=
\frac1{720}\, \sum_{\sigma\in \fS_6}X_k(\sigma)\,X_{s[\varpi]}(\sigma)
\]
where the last equality comes from the fact that the conjugacy classes 
of the symmetric are stable by inversion and $X_{s[\varpi]}$ denotes the character of $s[\varpi]$. For a partition $p$ of $6$, we put $C_p$ for the conjugacy class of an element of $\fS_6$ whose cycle shape corresponds with $p$ then, the last formula reads as 
\[
m_{{s[\varpi]}}(M_k(\Gamma[2])=
\frac1{720}\, \sum_{p}|C_{p}|\,X_k(s[p])\,X_{s[\varpi]}(s[p])
\]
where the sum is over the partition of $6$ and $g(s[p])$ stands for the evaluation of  a class function $g$ on a single element of the conjugacy class $C_p$.
Therefore we get
\begin{align*}
\sum_{k\geqslant 0}m_{{s[\varpi]}}(M_k(\Gamma[2])\, t^k =&
\frac1{720}\, \sum_{p}|C_{p}|\,X_{s[\varpi]}(s[p])
\Big(\sum_{k\geqslant 0} X_k(s[p])t^k\Big)\\
&=
\frac1{720}\, \sum_{p}|C_{p}|\,X_{s[\varpi]}(s[p])\,
\frac{(1+\text{sgn}(s[p])t^5)(1-t^8)}{f_{s[p]}(t^2)}.
\end{align*}
By using the character table of the group $\fS_6$ as given for example in \cite[p. 400]{Igusa64}, we deduce the theorem.
\end{proof}

\begin{remark}
A couple of sanity checks. Firstly,
\[
\sum_{k \geqslant 0}
\sum_{s[\varpi]}\dim (s[\varpi])
m_{{s[\varpi]}}(M_k(\Gamma[2]))\,t^k=\sum_{k \geqslant 0}\dim M_{k}(\Gamma[2])\,t^k=\frac{(1+t^2)(1+t^4)(1+t^5)}{(1-t^2)^4},
\]
which is in agreement with Theorem \ref{Igusa}. Secondly, the generating series 
for the dimension of spaces of modular forms on $\Gamma_{0}[2]$ is given by
(see \cite[Appendix A.1]{Roy_Schmidt_Yi} and the references therein)
\[
\sum_{k \geqslant 0}\dim M_{k}(\Gamma_0[2])\,t^k
=\frac{1+t^{19}}{(1-t^2)(1-t^4)^2(1-t^6)}.  
\]
We checked that by adding the generating series for the multiplicities of the irreducible representations $s[6], s[4,2]$ and $s[2^3]$ we recover this formula.
\end{remark}
Let us give the first few isotypical decompositions of $M_k(\Gamma[2])$, 
we put $d=\dim M_k(\Gamma[2])$

\begin{footnotesize}
\begin{center}
\[
\begin{array}{c|c|c|c|c|c|c|c|c|c|c|c|c}
s[\varpi] & s[6] & s[5,1] &  s[4,2] & s[4,1^2] & s[3^2] & s[3,2,1] & s[3,1^3] & s[2^3] & s[2^2,1^2] & s[2,1^4] & s[1^6] & \\
\dim s[\varpi]& 1& 5 & 9 & 10 & 5 & 16 & 10 & 5 & 9 & 5 & 1&\\
\hline
\hline
k & & & & & & & & & & & & d \\
\hline
0 & 1 & 0 & 0 & 0 & 0 & 0 & 0 & 0 & 0 & 0 & 0 & 1\\
1 & 0 & 0 & 0 & 0 & 0 & 0 & 0 & 0 & 0 & 0 & 0 & 0\\
2 & 0 & 0 & 0 & 0 & 0 & 0 & 0 & 1 & 0 & 0 & 0 & 5\\
3 & 0 & 0 & 0 & 0 & 0 & 0 & 0 & 0 & 0 & 0 & 0 & 0\\
4 &  1 & 0 & 1 & 0 & 0 & 0 & 0 & 1 & 0 & 0 & 0 & 15\\
5 & 0 & 0 & 0 & 0 & 0 & 0 & 0 & 0 & 0 & 0 & 1 & 1\\
6 &  1 & 0 & 1 & 0 & 0 & 0 & 1 & 2 & 0 & 1 & 0 & 35\\
7 &  0 & 0 & 0 & 0 & 1 & 0 & 0 & 0 & 0 & 0 & 0 & 5\\
8 &  1 & 0 & 3 & 0 & 0 & 1 & 1 & 3 & 0 & 0 & 0 & 69\\
9 &  0 & 0 & 0 & 0 & 1 & 0 & 0 & 0 & 1 & 0 & 1 & 15\\
10 &  2 & 0 & 3 & 0 & 0 & 2 & 3 & 4 & 0 & 2 & 0 & 121\\
11 & 0 & 1 & 0 & 1 & 2 & 0 & 0 & 0 & 1 & 0 & 1 & 35
\end{array}
\]
\end{center}
\end{footnotesize}

Next, we give the dimension of the various pieces of the space 
$M_{k}(\Gamma[2])$ as in (\ref{OrthoDecom}) and (\ref{DecompPackets})
and also their isotypical decomposition. We distinguish two cases according to the parity of $k$. 

\subsection{Isotypical decomposition of 
\texorpdfstring{$M_{2k+1}(\Gamma[2])$}{M {2k+1}(Gamma[2])}} 

We already have seen that $M_{2k+1}(\Gamma[2])=S_{2k+1}(\Gamma[2])$ and from Theorem \ref{Igusa} we deduce 
$\dim S_{1}(\Gamma[2])=\dim S_{3}(\Gamma[2])=0$ and
\[
\dim S_{2k+1}(\Gamma[2])=
(2k^3-9k^2+19k-15)/3 \quad \text{for} \quad k\geqslant 2.
\]
Therefore the generating series for $\dim S_{2k+1}(\Gamma[2])$ is given by
\[
\sum_{k\geqslant 0}
\dim S_{2k+1}(\Gamma[2])\, t^{2k+1}=
\frac{t^5(1+t^2+t^4+t^{6})}{(1-t^2)^4}.
\]
We have seen that for $k\geqslant 0$ we have
$
M_{2k+1}(\Gamma[2])=S_{2k+1}(\Gamma[2])=\chi_5\cdot M_{2k-4}(\Gamma[2])$
and since the cusp form $\chi_5$ is $\fS_6$-anti-invariant (i.e. it occurs in the alternating representation $s[1^6]$) we get
\[
\iso_{\fS_6}M_{2k+1}(\Gamma[2])=\iso_{\fS_6}S_{2k+1}(\Gamma[2])=
s[1^6] \otimes \iso_{\fS_6} M_{2k-4}(\Gamma[2]).
\]
As a corollary of Theorem \ref{Igusa2}, we deduce the following.  
\begin{cor}
The following table gives the generating series for the multiplicity of 
the irreducible representations of $\,\fS_6$ in $S_{2k+1}(\Gamma[2])$
\end{cor}
\begin{center}
\begin{Tabular}[1.35]{c|c}
$s[\varpi]$ & $\sum_{k \geqslant 0}m_{{s[\varpi]}}(S_{2k+1}(\Gamma[2]))\,t^{2k+1}$ \\
\hline
$s[6]$ & $\frac{t^{35}}{(1-t^4)(1-t^6)(1-t^{10})(1-t^{12})}$   \\
\hline
$s[5,1]$ & $\frac{t^{11}}{((1-t^4)(1-t^6))^2}$   \\
\hline
$s[4,2]$ & $\frac{t^{19}}{(1-t^2)(1-t^4)^2(1-t^{10})}$   \\
\hline
$s[4,1^2]$ & $\frac{t^{11}(1+t^4)}{(1-t^2)(1-t^4)(1-t^6)(1-t^{12})}$   \\
\hline
$s[3^2]$  & $\frac{t^7}{(1-t^2)(1-t^4)(1-t^6)(1-t^{12})}$   \\
\hline
$s[3,2,1]$ & $\frac{t^{13}(1+t^2+t^4+t^6)}{(1-t^2)(1-t^6)^2(1-t^{10})}$   \\
\hline
$s[3,1^3]$ & $\frac{t^{17}(1+t^4)}{(1-t^2)(1-t^4)(1-t^6)(1-t^{12})}$  \\
\hline
$s[2^3]$ & $\frac{t^{25}}{(1-t^2)(1-t^4)(1-t^6)(1-t^{12})}$   \\
\hline
$s[2^2,1^2]$ & $\frac{t^9}{(1-t^2)(1-t^4)^2(1-t^{10})}$  \\
\hline
$s[2,1^4]$ & $\frac{t^{17}}{((1-t^4)(1-t^6))^2}$ \\
\hline
$s[1^6]$ & $\frac{t^5}{(1-t^4)(1-t^6)(1-t^{10})(1-t^{12})}$  \\
\end{Tabular}
\end{center}
\begin{proof}
Let $G_{s[\varpi]}$ be the generating series of the multiplicity of the irreducible representation $s[\varpi]$ in $M_{k}(\Gamma[2])$. Then the generating series of the multiplicity of the irreducible representation $s[\varpi]$ in $M_{2k+1}(\Gamma[2])=S_{2k+1}(\Gamma[2])$ is given by 
$(G_{s[\varpi]}(t)-G_{s[\varpi]}(-t))/2$.
\end{proof}
\smallskip
Let us give the first few isotypical decompositions of 
$S_{2k+1}(\Gamma[2])$, we put $d=\dim S_{2k+1}(\Gamma[2])$.

\begin{footnotesize}
\begin{center}
\[
\begin{array}{c|c|c|c|c|c|c|c|c|c|c|c|c}
s[\varpi] & s[6] & s[5,1] &  s[4,2] & s[4,1^2] & s[3^2] & s[3,2,1] & s[3,1^3] & s[2^3] & s[2^2,1^2] & s[2,1^4] & s[1^6] & \\
\dim s[\varpi]& 1& 5 & 9 & 10 & 5 & 16 & 10 & 5 & 9 & 5 & 1&\\
\hline
\hline
2k+1 & & & & & & & & & & & & d \\
\hline
1 & 0 & 0 & 0 & 0 & 0 & 0 & 0 & 0 & 0 & 0 & 0 & 0\\
3 & 0 & 0 & 0 & 0 & 0 & 0 & 0 & 0 & 0 & 0 & 0 & 0\\
5 & 0 & 0 & 0 & 0 & 0 & 0 & 0 & 0 & 0 & 0 & 1 & 1\\
7 &  0 & 0 & 0 & 0 & 1 & 0 & 0 & 0 & 0 & 0 & 0 & 5\\
9 & 0 & 0 & 0 & 0 & 1 & 0 & 0 & 0 & 1 & 0 & 1 & 15\\
11 &  0 &  1 &  0 &  1 &  2 &  0 &  0 &  0 &  1 &  0 &  1 & 35\\
13 & 0 &  0 &  0 &  1 &  3 &  1 &  0 &  0 &  3 &  0 &  1 & 69\\
15 &  0 &  2 &  0 &  3 &  4 &  2 &  0 &  0 &  3 &  0 &  2 & 121
\end{array}
\]
\end{center}
\end{footnotesize}

We further refine these formulas according to the decomposition (\ref{DecompPackets}).
Conjecture 6.6 of \cite{BFvdG}, proved by R\"osner, see \cite[Section 5]{Roesner}, tells us that only the Arthur packets $\textbf{(P)}$ and $\textbf{(G)}$ can occur in $S_{2k+1}(\Gamma[2])$. Hence, 
\[
S_{2k+1}(\Gamma[2])=
S^{\textbf{(G)}}_{2k+1}(\Gamma[2])\oplus S^{\textbf{(P)}}_{2k+1}(\Gamma[2]).
\]
Moreover this result gives the isotypical decomposition of $S^{\textbf{(P)}}_{2k+1}(\Gamma[2])$:
\[
\iso_{\fS_6}
S^{\textbf{(P)}}_{2k+1}(\Gamma[2])=
d^{-}_{2,4k} s[5,1] +
d_{4,4k} s[3^2]  +
d^{+}_{2,4k}  s[1^6]
\]
where 
\begin{equation}\label{def_d_k}
d_{N,k}= \dim S_{k}(\Gamma_0(N))^{\text{new}} \quad 
\text{and} \quad
d^{\pm}_{N,k}=\dim S^{\pm}_{k}(\Gamma_0(N))^{\text{new}}.
\end{equation}
Here $\Gamma_0(N)$ denotes the Hecke congruence subgroup of level $N$ of $\SL(2,\Z)$ and 
$S^{\pm}_{k}(\Gamma_0(N))^{\text{new}}$ 
denotes the space of new cusp forms of weight $k$ on 
$\Gamma_0(N)$ with eigenvalues $\pm 1$ for the Fricke involution. 
Note that for $N=1$, we have $d_{1,k}= \dim S_{k}(\Gamma_0(1))^{\text{new}}=\dim S_k(\SL(2,\Z))$.
The dimension of the space $S_{k}(\Gamma_0(N))^{\text{new}}$ is classical, see \cite[Chapter 6]{Stein} and for example we have:
\begin{equation}\label{dimnew2new4}
\dim S_{4k}(\Gamma_0(2))^{\text{new}}=k-1-2\,\lfloor k/3 \rfloor
\quad
\text{and}
\quad
\dim S_{4k}(\Gamma_0(4))^{\text{new}}=\lfloor k/3 \rfloor
\quad
\text{for}
\quad 
k\geqslant 1.
\end{equation}
For $k>2$, the dimensions of $S^{\pm}_{k}(\Gamma_0(2))^{\text{new}}$ are given by (see \cite[Theorem 2.2]{MartinK})
\[
d^{\pm}_{2,k}=
\left\{
\begin{tabular}{ccc}
$(d_{2,k}\pm 1)/2$  & \text{if} & $k\equiv 0 \mod 8$\\
$(d_{2,k}\mp 1)/2$  & \text{if} & $k\equiv 2 \mod 8$\\
$d_{2,k}/2$ &  \text{if} & $k\equiv 4, 6 \mod 8$.
\end{tabular}
\right.
\]
So the generating series for the multiplicities of the irreducible representations $s[5,1]$, $s[3^2]$ and $s[1^6]$ in 
$S^{\textbf{(P)}}_{2k+1}(\Gamma[2])$ for $k\geqslant 0$ are given by

\begin{center}
\begin{Tabular}[1.3]{c|c|c|c}
$s[\varpi]$ & $s[5,1]$ & $s[3^2]$  & $s[1^6]$ \\
\hline
$\sum_{k \geqslant 0}m_{{s[\varpi]}}(S^{\textbf{(P)}}_{2k+1}(\Gamma[2]))\,t^{2k+1}$ & $\frac{t^{11}}{(1-t^4)(1-t^6)}$ & $\frac{t^7}{(1-t^2)(1-t^6)}$ & $\frac{t^5}{(1-t^4)(1-t^6)}$
\end{Tabular}
\end{center}

\smallskip
\noindent Keeping in mind that $\dim s[5,1]=\dim s[3^2]=5$ and $\dim s[1^6]=1$, we get
\[
\sum_{k\geqslant 0}
\dim S^{\textbf{(P)}}_{2k+1}(\Gamma[2]) t^{2k+1}=
\frac{t^5+5t^7+5t^9+5t^{11}}{(1-t^4)(1-t^6)}.
\]
From this we deduce the generating series for $\dim S^{\textbf{(G)}}_{2k+1}(\Gamma[2])$
\begin{align*}
\sum_{k\geqslant 0}
\dim S^{\textbf{(G)}}_{2k+1}(\Gamma[2])\,t^{2k+1}=&
\sum_{k\geqslant 0}
\bigl(\dim S_{2k+1}(\Gamma[2])-\dim S^{\textbf{(P)}}_{2k+1}(\Gamma[2])\bigr)\,t^{2k+1}
=
\frac{t^9(t^8-2t^6+10t^4+6t^2+9)}{(1-t^2)^3(1-t^6)(1+t^2)}.
\end{align*}
We also deduce the generating series for the multiplicities of the irreducible representations 
in $S^{\textbf{(G)}}_{2k+1}(\Gamma[2])$ for $k\geqslant 0$

\begin{itemize}
\item for 
$
s[\varpi]\in 
\left\{
s[6], s[4,2], s[4,1^2], s[3,2,1], s[3,1^3], s[2^3], s[2^2,1^2], s[2,1^4]
\right\}
$
we have
\[
m_{{s[\varpi]}}(S^{\textbf{(G)}}_{2k+1}(\Gamma[2]))=m_{{s[\varpi]}}(S_{2k+1}(\Gamma[2]))
\]
\item for $s[\varpi] \in \left\{
s[5,1], s[3^2], s[1^6]
\right\}$
we have
\[
m_{{s[\varpi]}}(S^{\textbf{(G)}}_{2k+1}(\Gamma[2]))=
m_{{s[\varpi]}}(S_{2k+1}(\Gamma[2]))-
m_{{s[\varpi]}}(S^{\textbf{(P)}}_{2k+1}(\Gamma[2]))
\]
\end{itemize}

so
\begin{center}
\begin{Tabular}[1.35]{c|c|c|c}
$s[\varpi]$ & $s[5,1]$ & $s[3^2]$ & $s[1^6]$\\
\hline
$\sum_{k \geqslant 0}m_{{s[\varpi]}}(S^{\textbf{(G)}}_{2k+1}(\Gamma[2]))\,t^{2k+1}$   
& $\frac{t^{15}(1+t^2-t^6)}{((1-t^4)(1-t^6))^2}$
& $\frac{t^{11}(1+t^{8}-t^{12})}{(1-t^2)(1-t^4)(1-t^6)(1-t^{12})}$
& $\frac{t^{15}(1+t^{2}-t^{12})}{(1-t^4)(1-t^6)(1-t^{10})(1-t^{12})}$
\end{Tabular}
\end{center}

The next table gives the first few isotypical decompositions of $S_{2k+1}(\Gamma[2])$, we indicate the
multiplicities of Saito-Kurakawa lifts in blue and those of the general type in black, we put 
${\color{blue} d_P}=\dim S^{\textbf{(P)}}_{k}(\Gamma[2])$ 
and $d_G=\dim S^{\textbf{(G)}}_{k}(\Gamma[2])$

\begin{footnotesize}
\begin{center}
\[
\begin{array}{c|c|c|c|c|c|c|c|c|c|c|c|c}
s[\varpi] & s[6] & s[5,1] &  s[4,2] & s[4,1^2] & s[3^2] & s[3,2,1] & s[3,1^3] & s[2^3] & s[2^2,1^2] & s[2,1^4] & s[1^6] & \\
\dim s[\varpi]& 1& 5 & 9 & 10 & 5 & 16 & 10 & 5 & 9 & 5 & 1&\\
\hline
\hline
2k+1 & & & & & & & & & & & & d_G+{\color{blue} d_P} \\
\hline
1 & 0 & 0 & 0 & 0 & 0 & 0 & 0 & 0 & 0 & 0 & 0 & 0+{\color{blue} 0}\\
3 & 0 & 0 & 0 & 0 & 0 & 0 & 0 & 0 & 0 & 0 & 0 & 0+{\color{blue} 0}\\
5 & 0 & 0 & 0 & 0 & 0 & 0 & 0 & 0 & 0 & 0 & {\color{blue} 1} & 0+{\color{blue} 1} \\
7 &  0 & 0 & 0 & 0 & {\color{blue} 1} & 0 & 0 & 0 & 0 & 0 & 0 & 0+ {\color{blue} 5}\\
9 & 0 & 0 & 0 & 0 &  {\color{blue} 1} & 0 & 0 & 0 & 1 & 0 &  {\color{blue} 1} & 9+ {\color{blue} 6}\\
11 &  0 &   {\color{blue} 1} &  0 &  1 &  1+ {\color{blue} 1} &  0 &  0 &  0 &  1 &  0 &   {\color{blue} 1} & 24+ {\color{blue} 11}\\
13 & 0 &  0 &  0 &  1 &  1+ {\color{blue} 2} &  1 &  0 &  0 &  3 &  0 &   {\color{blue} 1} & 58+ {\color{blue} 11}\\
15 &  0 &  1+ {\color{blue} 1} &  0 &  3 &  2+ {\color{blue} 2} &  2 &  0 &  0 &  3 &  0 &  1+ {\color{blue} 1} & 105+ {\color{blue} 16}
\end{array}
\]
\end{center}
\end{footnotesize}

\subsection{Isotypical decomposition of 
\texorpdfstring{$M_{2k}(\Gamma[2])$}{M {2k}(Gamma[2])}} 

We start by giving the  generating series for the multiplicities of the irreducible representations of $\mathfrak{S}_6$ in $M_{2k}(\Gamma[2])$.
As a corollary of Theorem \ref{Igusa2}, we deduce the following. 
\begin{cor}
The following table gives the generating series for the multiplicity of 
the irreducible representations of $\,\fS_6$ in $M_{2k}(\Gamma[2])$
\end{cor}
\begin{center}
\begin{Tabular}[1.3]{c|c}
$s[\varpi]$ & $\sum_{k \geqslant 0}m_{{s[\varpi]}}(M_{2k}(\Gamma[2]))\,t^{2k}$ \\
\hline
$s[6]$ & $\frac{1}{(1-t^4)(1-t^6)(1-t^{10})(1-t^{12})}$   \\
\hline
$s[5,1]$ & $\frac{t^{12}}{((1-t^4)(1-t^6))^2}$   \\
\hline
$s[4,2]$ & $\frac{t^4}{(1-t^2)(1-t^4)^2(1-t^{10})}$   \\
\hline
$s[4,1^2]$ & $\frac{t^{12}(1+t^4)}{(1-t^2)(1-t^4)(1-t^6)(1-t^{12})}$   \\
\hline
$s[3^2]$  & $\frac{t^{20}}{(1-t^2)(1-t^4)(1-t^6)(1-t^{12})}$   \\
\hline
$s[3,2,1]$ & $\frac{t^8(1+t^2+t^4+t^8)}{(1-t^2)(1-t^6)^2(1-t^{10})}$   \\
\hline
$s[3,1^3]$ & $\frac{t^6(1+t^4)}{(1-t^2)(1-t^4)(1-t^6)(1-t^{12})}$  \\
\hline
$s[2^3]$ & $\frac{t^2}{(1-t^2)(1-t^4)(1-t^6)(1-t^{12})}$   \\
\hline
$s[2^2,1^2]$ & $\frac{t^{14}}{(1-t^2)(1-t^4)^2(1-t^{10})}$  \\
\hline
$s[2,1^4]$ & $\frac{t^6}{((1-t^4)(1-t^6))^2}$ \\
\hline
$s[1^6]$ & $\frac{t^{30}}{(1-t^4)(1-t^6)(1-t^{10})(1-t^{12})}$  \\
\end{Tabular}
\end{center}
\begin{proof}
Let $G_{s[\varpi]}$ be the generating series of the multiplicity of the irreducible representation $s[\varpi]$ in $M_{k}(\Gamma[2])$. Then the generating series of the multiplicity of the irreducible representation $s[\varpi]$ in $M_{2k}(\Gamma[2])$ is given by 
$(G_{s[\varpi]}(t)+G_{s[\varpi]}(-t))/2$.
\end{proof}

\begin{remark}
As a sanity check, we verify
\[
\sum_{k \geqslant 0}
\sum_{s[\varpi]}\dim (s[\varpi])
m_{{s[\varpi]}}(M_{2k}(\Gamma[2]))\,t^{2k}=\frac{(1+t^2)(1+t^4)}{(1-t^2)^4}
=\sum_{k \geqslant 0}\dim M_{2k}(\Gamma[2])\,t^{2k}
\]
in agreement with Theorem \ref{Igusa}.
\end{remark}
Let us give the first few isotypical decompositions of $M_{2k}(\Gamma[2])$, we put $d=\dim M_{2k}(\Gamma[2])$

\begin{footnotesize}
\begin{center}
\[
\begin{array}{c|c|c|c|c|c|c|c|c|c|c|c|c}
s[\varpi] & s[6] & s[5,1] &  s[4,2] & s[4,1^2] & s[3^2] & s[3,2,1] & s[3,1^3] & s[2^3] & s[2^2,1^2] & s[2,1^4] & s[1^6] & \\
\dim s[\varpi]& 1& 5 & 9 & 10 & 5 & 16 & 10 & 5 & 9 & 5 & 1&\\
\hline
\hline
k & & & & & & & & & & & & d \\
\hline
0 & 1 & 0 & 0 & 0 & 0 & 0 & 0 & 0 & 0 & 0 & 0 & 1\\
2 & 0 & 0 & 0 & 0 & 0 & 0 & 0 & 1 & 0 & 0 & 0 & 5\\
4 &  1 & 0 & 1 & 0 & 0 & 0 & 0 & 1 & 0 & 0 & 0 & 15\\
6 &  1 & 0 & 1 & 0 & 0 & 0 & 1 & 2 & 0 & 1 & 0 & 35\\
8 &  1 & 0 & 3 & 0 & 0 & 1 & 1 & 3 & 0 & 0 & 0 & 69\\
10 &  2 & 0 & 3 & 0 & 0 & 2 & 3 & 4 & 0 & 2 & 0 & 121\\
12 & 3 & 1 & 6 & 1 & 0 & 3 & 4 & 5 & 0 & 2 & 0 & 195\\
14 & 2 & 0 & 7 & 1 & 0 & 6 & 6 & 8 & 1 & 3 & 0 & 295\\
16 & 4 & 2 & 11 & 3 & 0 & 8 & 8 & 9 & 1 & 4 & 0 & 425
\end{array}
\]
\end{center}
\end{footnotesize}

From (\ref{OrthoDecom}) we have the following decomposition of the space $M_{2k}(\Gamma[2])$: 
\[
M_{2k}(\Gamma[2])=E_{2k}(\Gamma[2]) \oplus S_{2k}(\Gamma[2])
\quad
\text{for}
\quad k\geqslant 0
\]
so we are going to give  the isotypical decomposition of the Eisenstein and the cuspidal parts according to (\ref{DecompPackets}).

\subsubsection{\textbf{Isotypical decomposition of $E_{2k}(\Gamma[2])$} }\label{EevenQ} 
From (\ref{DecompPackets}), we have 
\[
E_{2k}(\Gamma[2])=SE_{2k}(\Gamma[2])\oplus KE_{2k}(\Gamma[2])
\quad
\text{for}
\quad k\geqslant 0.
\]
Recall that  $SE_{2k}(\Gamma[2])$ corresponds to Siegel-Eisenstein series 
of weight $2k$ while
$KE_{2k}(\Gamma[2])$ to Klingen-Eisenstein series of weight $2k$. From  \cite[Section 13]{CvdGG}, 
we deduce that
\[
\iso_{\fS_6}SE_{0}(\Gamma[2])=s[6], \; 
\iso_{\fS_6}SE_{2}(\Gamma[2])=s[2^3] 
\text{ and } 
\iso_{\fS_6}SE_{2k}(\Gamma[2])=s[6]+ s[4,2] + s[2^3]  \text{\, for\, }  k\geqslant 2.
\]
By construction, Klingen-Eisenstein series of weight $2k$ for $k>2$ (to ensure convergence) on $\Gamma[2]$ come 
from cusp forms of weight $2k$ on $\Gamma(2)$. Since $\SL(2,\Z)/\Gamma(2)\cong \mathfrak{S}_3$, 
the space $S_{2k}(\Gamma(2))$ can be decomposed into irreducible representations of 
$\mathfrak{S}_3$ and this is given by
\begin{equation}\label{DecomCuspLevel2}
\iso_{\fS_3}S_{2k}(\Gamma(2))=d_{1,2k}s[3]+(d_{1,2k}+d_{2,2k})s[2,1]+d_{4,2k}s[1^3],
\end{equation}
where the integers $d_{N,k}$ are defined as in (\ref{def_d_k}). Note that the previous formula 
corrects \cite[Proposition 6.1]{Petersen}.
Recall that the dimension of the space $S_{2k}(\Gamma(2)) \cong S_{2k}(\Gamma_0(4))$ 
is $k-2$ for $k\geqslant 3$ and $0$ otherwise.

From the previous table, Theorem \ref{Igusa} and the isotypical decomposition of $SE_{2k}(\Gamma[2])$,  
we deduce that 
$
\iso_{\fS_6}KE_{0}(\Gamma[2])
=\iso_{\fS_6}KE_{2}(\Gamma[2])
=0.
$
For $k\geqslant 2$, Proposition 13.1 of \cite{CvdGG} gives 
\[
\iso_{\fS_6}KE_{2k}(\Gamma[2])= \text{Ind}_{H}^{\mathfrak{S}_6}\left(\iso_{\fS_3}S_{2k}(\Gamma(2))\right)
\]
where $H$ denotes the stabiliser in $\mathfrak{S}_6$  of one of the $1$-dimensional boundary components of the Satake compactification of $\Gamma[2] \backslash\mathfrak{H}_2$. 
Note that $H$ is of order 48 and recall (for more details see \cite[Section 2]{CvdGG}) that we have
\begin{center}
\begin{tabular}{c|c|c|c}
$s[\varpi]$ & $s[3]$ & $s[2,1]$ & $s[1^3]$ \\
\hline
${\rm Ind}_{H}^{\mathfrak{S}_6}\left(s[\varpi] \right)$ & 
$s[6]\oplus s[5,1]\oplus s[4,2]$ &
$s[4,2]\oplus s[3,2,1]\oplus s[2^3]$ &
$s[3,1^3]\oplus s[2,1^4]$
\end{tabular}
\end{center}
Therefore, for $k\geqslant 2$, the isotypical decomposition of the space
$KE_{2k}(\Gamma[2])$ is as follows
\begin{align*}
\iso_{\fS_6}KE_{2k}(\Gamma[2])=&
d_{1,2k}(s[6] + s[5,1])+
(2d_{1,2k}+d_{2,2k})s[4,2] +
(d_{1,2k}+d_{2,2k})(s[3,2,1]+s[2^3])\\ 
&+  
d_{4,2k}(s[3,1^3]+s[2,1^4]).
\end{align*}
Putting this together we get the generating series  for the multiplicities of the irreducible 
representations of $\mathfrak{S}_6$ in $SE_{2k}(\Gamma[2])$ and $KE_{2k}(\Gamma[2])$

\smallskip

\begin{center}
\begin{Tabular}[1.25]{c|c|c}
$s[\varpi]$ & 
$\sum_{k \geqslant 0}m_{{s[\varpi]}}(SE_{2k}(\Gamma[2]))\,t^{2k}$ & 
$\sum_{k \geqslant 0}m_{{s[\varpi]}}(KE_{2k}(\Gamma[2]))\,t^{2k}$ \\
\hline
$s[6]$ & $\frac{1-t^2+t^4}{1-t^2}$ & $\frac{t^{12}}{(1-t^4)(1-t^6)}$   \\
\hline
$s[5,1]$ & $0$  & $\frac{t^{12}}{(1-t^4)(1-t^6)}$  \\
\hline
$s[4,2]$ & $\frac{t^4}{1-t^2}$  & $\frac{t^{8}}{(1-t^2)(1-t^4)}$  \\
\hline
$s[3,2,1]$ & $0$ & $\frac{t^{8}}{(1-t^2)(1-t^6)}$  \\
\hline
$s[3,1^3]$ & $0$ & $\frac{t^{6}}{(1-t^4)(1-t^6)}$ \\
\hline
$s[2^3]$ &  $\frac{t^2}{1-t^2}$  & $\frac{t^{8}}{(1-t^2)(1-t^6)}$  \\
\hline
$s[2,1^4]$ & $0$ & $\frac{t^{6}}{(1-t^4)(1-t^6)}$ \\
\end{Tabular}
\end{center}
and $0$ for $s[4,1^2], s[3^2], s[2^2,1^2]$ and $s[1^6]$.
\begin{remark}
The isotypical decomposition of the space $S_{2k}(\Gamma(2))$ in  \cite[Proposition 13.1]{CvdGG} was written as
\[
\iso_{\fS_3}S_{2k}(\Gamma(2))=\text{{Sym}}^k(s[2,1])-
\left\{
\begin{tabular}{ccc}
$s[2,1]$  & \text{if} & $k=1$\\
$s[3]+s[2,1]$ &  \text{if} & $k\geqslant 2$.
\end{tabular}
\right.
\]
This directly gives $\dim S_{4k}(\Gamma(2))=2(k-1)$ for $k\geqslant 1$, this can also be checked by using 
(\ref{dimnew2new4}), (\ref{DecomCuspLevel2}) and $\dim S_{4k}(\SL(2,\Z))=\lfloor k/3 \rfloor$.
As a sanity check, we verify
\begin{align*}
\sum_{k \geqslant 0}
\sum_{s[\varpi]}\dim (s[\varpi])
m_{{s[\varpi]}}(KE_{2k}(\Gamma[2]))t^{2k}  &=  
15\frac{t^6}{(1-t^2)^2}=15 \sum_{k \geqslant 0} \dim S_{2k}(\Gamma(2))t^{2k}
\end{align*}
in agreement with $KE_{2k}(\Gamma[2]) \cong S_{2k}(\Gamma(2))^{\oplus 15}$.
\end{remark}

The next table gives the first few isotypical decompositions of $E_{2k}(\Gamma[2])$ where we indicate the
multiplicities of the Siegel-Eisenstein part in blue and those of the Klingen-Eisenstein part in black, we put 
${\color{blue} d_F}=\dim SE_{k}(\Gamma[2])$ 
and $d_Q=\dim KE_{k}(\Gamma[2])$

\begin{footnotesize}
\begin{center}
\[
\begin{array}{c|c|c|c|c|c|c|c|c|c|c|c|c}
s[\varpi] & s[6] & s[5,1] &  s[4,2] & s[4,1^2] & s[3^2] & s[3,2,1] & s[3,1^3] & s[2^3] & s[2^2,1^2] & s[2,1^4] & s[1^6] & \\
\dim s[\varpi]& 1& 5 & 9 & 10 & 5 & 16 & 10 & 5 & 9 & 5 & 1&\\
\hline
\hline
k & & & & & & & & & & & & d={\color{blue} d_F}+d_Q \\
\hline
0 & {\color{blue} 1} & 0 & 0 & 0 & 0 & 0 & 0 & 0 & 0 & 0 & 0 & {\color{blue} 1}+0\\
2 & 0 & 0 & 0 & 0 & 0 & 0 & 0 & {\color{blue} 1} & 0 & 0 & 0 & {\color{blue} 5}+0\\
4 &  {\color{blue} 1} & 0 & {\color{blue} 1} & 0 & 0 & 0 & 0 & {\color{blue} 1} & 0 & 0 & 0 & {\color{blue} 15}+0\\
6 &  {\color{blue} 1} & 0 & {\color{blue} 1} & 0 & 0 & 0 & 1 & {\color{blue} 1} & 0 & 1 & 0 & {\color{blue} 15}+ 15\\
8 &  {\color{blue} 1} & 0 & {\color{blue} 1}+1 & 0 & 0 & 1 & 0 & {\color{blue} 1}+1 & 0 & 0 & 0 &  {\color{blue} 15}+ 30\\
10 &  {\color{blue} 1} & 0 & {\color{blue} 1}+1 & 0 & 0 & 1 & 1 & {\color{blue} 1}+1 & 0 & 1 & 0 &  {\color{blue} 15}+ 45\\
12 & {\color{blue} 1}+1 & 1 & {\color{blue} 1}+2 & 0 & 0 & 1 & 1 & {\color{blue} 1}+1 & 0 & 1 & 0 &  {\color{blue} 15}+ 60\\
14 & {\color{blue} 1} & 0 & {\color{blue} 1}+2 & 0 & 0 & 2 & 1 & {\color{blue} 1}+2 & 0 & 1 & 0 &  {\color{blue} 15} + 75\\
\end{array}
\]
\end{center}
\end{footnotesize}

\subsubsection{\textbf{Isotypical decomposition of $S_{2k}(\Gamma[2])$}} 
For $k\geqslant 2$, we know, see \cite[ pp. 882-883]{Tsushima1982}, that 
\begin{equation}\label{Dim_S2k}
\dim S_{2k}(\Gamma[2])=\dim M_{2k}(\Gamma[2]) -15(k-2)-15=(k-2)(2k^2+7k-24)/3.
\end{equation}
We also know that 
$
S_{0}(\Gamma[2])=S_{2}(\Gamma[2])=\{0\}$. 
The generating series for the dimension of the spaces $S_{2k}(\Gamma[2])$ is therefore given by
\[
\sum_{k \geqslant 0}\dim S_{2k}(\Gamma[2])\,t^{2k}=
\frac{t^6(5+4t^2-5t^4)}{(1-t^2)^4}.
\]
By definition of cusp forms, for $k \geqslant 0$ we have
\[
m_{{s[\varpi]}}(S_{2k}(\Gamma[2]))=
m_{{s[\varpi]}}(M_{2k}(\Gamma[2]))-
\Big(
m_{{s[\varpi]}}(SE_{2k}(\Gamma[2]))+
m_{{s[\varpi]}}(KE_{2k}(\Gamma[2]))
\Big).
\]
So the generating series for the multiplicities of the irreducible representations of $\mathfrak{S}_6$
in  $S_{2k}(\Gamma[2])$ are given by 
\smallskip

\begin{center}
\begin{Tabular}[1.35]{c|c}
$s[\varpi]$ & $\sum_{k \geqslant 0}m_{{s[\varpi]}}(S_{2k}(\Gamma[2]))\,t^{2k}$ \\
\hline
$s[6]$ & $\frac{t^{10}(1+t^2-t^{12})}{(1-t^4)(1-t^6)(1-t^{10})(1-t^{12})}$   \\
\hline
$s[5,1]$ & $\frac{t^{16}(1+t^2-t^6)}{((1-t^4)(1-t^6))^2}$   \\
\hline
$s[4,2]$ & $\frac{t^8(1+t^6-t^{10})}{(1-t^2)(1-t^4)^2(1-t^{10})}$   \\
\hline
$s[4,1^2]$ & $\frac{t^{12}(1+t^4)}{(1-t^2)(1-t^4)(1-t^6)(1-t^{12})}$   \\
\hline
$s[3^2]$  & $\frac{t^{20}}{(1-t^2)(1-t^4)(1-t^6)(1-t^{12})}$   \\
\hline
$s[3,2,1]$ & $\frac{t^{10}(1+t^2+2t^4+t^8-t^{14})}{(1-t^2)(1-t^6)^2(1-t^{10})}$   \\
\hline
$s[3,1^3]$ & $\frac{t^8(1+t^2+t^{10}-t^{12})}{(1-t^2)(1-t^4)(1-t^6)(1-t^{12})}$  \\
\hline
$s[2^3]$ & $\frac{t^6(1+t^8-t^{12})}{(1-t^2)(1-t^4)(1-t^6)(1-t^{12})}$   \\
\hline
$s[2^2,1^2]$ & $\frac{t^{14}}{(1-t^2)(1-t^4)^2(1-t^{10})}$  \\
\hline
$s[2,1^4]$ & $\frac{t^{10}(1+t^2-t^6)}{((1-t^4)(1-t^6))^2}$ \\
\hline
$s[1^6]$ & $\frac{t^{30}}{(1-t^4)(1-t^6)(1-t^{10})(1-t^{12})}$  \\
\end{Tabular}
\end{center}
\medskip

Conjecture 6.6, proved by R\"osner 
(see \cite[Section 5]{Roesner}), tells us that
only the Arthur packets $\textbf{(P)}$ and $\textbf{(G)}$ can occur in $S_{2k}(\Gamma[2])$ and so 
\[
S_{2k}(\Gamma[2])=
S^{\textbf{(G)}}_{2k}(\Gamma[2])\oplus S^{\textbf{(P)}}_{2k}(\Gamma[2]).
\]
Moreover Conjecture 6.6 of \cite{BFvdG} gives the isotypical decomposition of $S^{\textbf{(P)}}_{2k}(\Gamma[2])$:
\begin{align*}
\iso_{\fS_6}S^{\textbf{(P)}}_{2k}(\Gamma[2])=&
d_{1,4k-2}\,s[6] +
(d_{1,4k-2}+d^+_{2,4k-2}) \, s[4,2] +
(d_{2,4k-2}+d^{-}_{2,4k-2}) \, s[2^3] 
\end{align*}
where the integers $d_{N,k}$ and $d^{\pm}_{N,k}$ are defined as in (\ref{def_d_k}).
The generating series for the multiplicity of the irreducible representations $s[6]$, $s[4,2]$ and $s[2^3]$ in 
$S^{\textbf{(P)}}_{2k}(\Gamma[2])$ for $k\geqslant 0$ are therefore given by

\begin{center}
\begin{Tabular}[1.35]{c|c|c|c}
$s[\varpi]$ & $s[6]$ & $s[4,2]$ & $s[2^3]$\\
\hline
$\sum_{k \geqslant 0}m_{{s[\varpi]}}(S^{\textbf{(P)}}_{2k}(\Gamma[2]))\,t^{2k}$ 
& $\frac{t^{10}}{(1-t^2)(1-t^{6})}$
& $\frac{t^{8}}{(1-t^2)(1-t^4)}$ 
& $\frac{t^{6}}{(1-t^2)(1-t^4)}$ 
\end{Tabular}
\end{center}
\smallskip
Keeping in mind that $\dim s[6]=1$, $\dim s[4,2]=9$ and $\dim s[2^3]=5$, we deduce
\[
\sum_{k\geqslant 0}
\dim S^{\textbf{(P)}}_{2k}(\Gamma[2]) t^{2k}=
\frac{t^6(5+14t^2+15t^4+10t^6)}{(1-t^4)(1-t^{6})}.
\]
From this we get the generating series for $\dim S^{\textbf{(G)}}_{2k}(\Gamma[2])$:
\[
\sum_{k\geqslant 0}
\dim S^{\textbf{(G)}}_{2k}(\Gamma[2]) t^{2k}=
\sum_{k\geqslant 0}
\bigl(\dim S_{2k}(\Gamma[2])-\dim S^{\textbf{(P)}}_{2k}(\Gamma[2])\bigr) t^{2k}\\
=
\frac{t^8(10+21t^2+9t^4-t^6-15t^8)}{(1-t^2)^2(1-t^4)(1-t^{6})}
\]
and also the generating series for the multiplicity of the irreducible representations 
in $S^{\textbf{(G)}}_{2k}(\Gamma[2])$ for $k\geqslant 0$
\begin{itemize}
\item for 
$
s[\varpi]\in \left\{s[5,1], s[4,1^2], s[3^2], s[3,2,1], s[3,1^3], s[2^2,1^2], s[2,1^4], s[1^6]\right\}
$
we have
\[
m_{{s[\varpi]}}(S^{\textbf{(G)}}_{2k}(\Gamma[2]))=m_{{s[\varpi]}}(S_{2k}(\Gamma[2]))
\]
\item  for 
$s[\varpi] \in \left\{s[6], s[4,2], s[2^3]\right\}$
we have
\[
m_{{s[\varpi]}}(S^{\textbf{(G)}}_{2k}(\Gamma[2]))=
m_{{s[\varpi]}}(S_{2k}(\Gamma[2]))-
m_{{s[\varpi]}}(S^{\textbf{(P)}}_{2k}(\Gamma[2]))
\]
\end{itemize}
so

\begin{center}
\begin{Tabular}[1.35]{c|c|c|c}
$s[\varpi]$ & $s[6]$ & $s[4,2]$ & $s[2^3]$\\
\hline
$\sum_{k \geqslant 0}m_{{s[\varpi]}}(S^{\textbf{(G)}}_{2k}(\Gamma[2]))\,t^{2k}$
& $\frac{t^{20}(1+t^2+t^4-t^{12}-t^{14})}{(1-t^4)(1-t^6)(1-t^{10})(1-t^{12})}$
& $\frac{t^{12}(1+t^2-t^{10})}{(1-t^2)(1-t^4)^2(1-t^{10})}$ 
& $\frac{t^{12}(1+t^2-t^{12})}{(1-t^2)(1-t^4)(1-t^6)(1-t^{12})}$
\end{Tabular}
\end{center}

\section{Euler characteristics of local systems} \label{sec-euler} 
Let $\AA_2[2]$ be the moduli space of principally polarized abelian surfaces equipped with a full level two structure. This is a smooth Deligne-Mumford stack 
defined over ${\rm Spec}(\Z[1/2])$. The space $\AA_2[2]$ comes equipped with a natural action of the symmetric group $\fS_6 \cong {\rm GSp}(4,\Z/2)$. 
Let $\pi: \mathcal X \to \AA_2[2]$ denote the universal object and define the local system $\V=R^1\pi_* \C$ on $(\AA_2[2])_{\C}$. To each pair of integers $(l,m)$, with $l \geq m \geq 0$, we get a local system $\V_{l,m}$ from the corresponding irreducible representation of $\mathrm{GSp}(4)$. 

The moduli space $\AA_2[2]$ can be identified with the disjoint union $\mathcal M_2[2] \sqcup \AA_{1,1}[2]$, where $\mathcal M_2[2]$ is the moduli space of tuples $(C,r_1,\ldots,r_6)$ where $C$ is a genus $2$ curve (note that curves are assumed to be projective, irreducible and smooth), and $\AA_{1,1}[2]$, is the moduli space of tuples $(C,r_1,\ldots,r_6)$ where $C$ is an unordered pair of elliptic curves intersecting in the point at infinity, and in both cases $(r_1,\ldots,r_6)$ is a $6$-tuple of marked Weierstra\ss\, points (distinct from infinity, in the elliptic curve case).
With this identification, the action of $\fS_6$ is by permutation of the marked Weierstra\ss\, points.
For more details about the above cf. \cite{BFvdG}.

 Let $\psi: (\AA_2[2])_{\C} \to A_2[2]$ denote the coarse moduli space and put $V_{l,m}=\psi_*\V_{l,m}$. Similarly we have the coarse moduli spaces $M_2[2]$ and $A_{1,1}[2]$ and by abuse of notation, $V_{l,m}$ will also denote the restriction of $V_{l,m}$ on $A_2[2]$ to any of these subspaces. 
There is an induced action of $\fS_6$ on the compactly supported Betti cohomology groups $H^i_c$ of these spaces with coefficients in $V_{l,m}$. 

We will now identify the representation ring $\Z[\fS_6]$ of $\fS_6$ with the ring of symmetric polynomials. With this interpretation, for a partition $\varpi$ of $6$, $s[\varpi]$ equals the corresponding Schur polynomial. Let also $p_i$ denote the $i$th power sum polynomial and put $p_{\varpi}=p_1^{\varpi_1}\cdots p_6^{\varpi_6}$.  
Moreover, for any $\lambda=(l,m)$, let $s_{<\lambda>}$ denote the symplectic Schur polynomial in four variables associated to $\lambda$, see \cite[Appendix A]{FH}.

\subsection{Formulas for the Euler characteristics} 
The aim of this section is to give a formula, for any $\lambda=(l,m)$, of the $\fS_6$-equivariant Euler characteristic, 
\[E_c(A_2[2],V_{\lambda})= \sum_{\varpi \vdash 6} E_{c,\varpi}(A_2[2],V_{\lambda}) s[\varpi]  \in \Z[\fS_6],
\]
where 
\[
E_{c,\varpi}(A_2[2],V_{\lambda})=\sum^4_{i=0} (-1)^i m_{s[\varpi]}\bigl(H^i_c(A_2[2],V_{\lambda})\bigr)
 \in \Z.
\]

Stratify the spaces $X_1=M_2[2]$ and $X_2=A_{1,1}[2]$ (or equivalently $M_2$ and $A_{1,1}$, the corresponding coarse moduli spaces without a level two structure), into strata $\Sigma_i(G)$, for $G$ a finite group, consisting of the curves corresponding to points of $X_i$ whose automorphism group equals $G$. Let $E_c(\Sigma_i(G))$ denote 
the Euler characteristic of $\Sigma_i(G)$. Say that $g \in G$ has eigenvalues $\xi_1(g)$, $\xi_2(g)$, $\xi_3(g)$ and $\xi_4(g)$ when acting on $H^1(C,\C)$ of a curve $C \in \Sigma_i(G)$. Say furthermore that the induced action of $g \in G$ on the six Weierstra\ss\, points of a curve $C \in \Sigma_i(G)$ has $\mu_j$ cycles of length $j$ for $j=1,\ldots,6$, giving a partition $\mu(g,G,i)$. Note that this data will be constant on the strata, i.e. independent of the choice of $C \in  \Sigma_i(G)$. On a strata $\Sigma_i(G)$ the Euler characteristic $E_c(\Sigma_i(G),V_{\lambda})=E_c(\Sigma_i(G)) \cdot \dim V_{\lambda}^{G}$ and hence 
\begin{equation} 
  E_c(A_2,V_{\lambda})=\sum_{i=1}^2 \sum_{G} \frac{E_c \bigl(\Sigma_i(G) \bigr)}{|G|}  \sum_{g \in G}
  s_{<\lambda>} \bigl(\xi_1(g), \xi_2(g),\xi_3(g),\xi_4(g) \bigr) \in \Z.
  \end{equation}
This method was used in \cite{Getzler} to find a formula for $E_c(M_2,V_{\lambda})$ for any $\lambda$. Adding the level two structure we need to take the action of $\fS_6$ on the Weierstra\ss\, points into account and one finds that,
\begin{equation} \label{eq-numeric}
  E_c(A_2[2],V_{\lambda})=\sum_{i=1}^2 \sum_{G} \frac{E_c \bigl(\Sigma_i(G) \bigr)}{|G|}  \sum_{g \in G}
  s_{<\lambda>} \bigl(\xi_1(g), \xi_2(g),\xi_3(g),\xi_4(g) \bigr) p_{\mu(g,G,i)} \in \Z[\fS_6].
  \end{equation}
This formula can be compared to the one in \cite[Section 9]{BvdG}.

In the two following sections, we will describe how to find the necessary information to compute \eqref{eq-numeric} for any $\lambda$.

\subsection{Smooth curves of genus two} 
The stratification by automorphism group $G$ for $M_2$ was found by Bolza \cite{Bolza}, see below. We follow the description in \cite[Section 4]{Getzler}. 

Curves $C$ of genus $2$ are described by equations $C_f:y^2-f(x)=0$, where $f$ square-free polynomial of degree $5$ or $6$.
The automorphism group $G_f$ of a curve $C_f$ is equal to the subgroup of $\mathrm{SL}(2,\C) \times \C^{\times}$ consisting of elements 
\[
(\gamma,u)=\Bigl(\Bigl( \begin{array}{cc} a & b \\ c & d \end{array} \Bigr),u\Bigr) \in \mathrm{SL}(2,\C) \times \C^{\times}
\quad \text{such that} \quad
f(x)=(\gamma,u) \cdot f(x)=\frac{(cx+d)^6}{u^2} \, f\Bigl(\frac{ax+b}{cx+d}\Bigr)
\]
quotiented by the subgroup generated by the element $(-\mathrm{id},-1) \in \mathrm{SL}(2,\C) \times \C^{\times} $.
These groups will be given as pairs $(\Gamma_f,\rho_f)$, where $\Gamma_f$ is a subgroup of $\mathrm{SL}(2,\C)$ that preserves the set of roots of $f$ and $\rho_f$ is a character of $\Gamma_f$ such that $G_f \cong \Gamma_f(\rho_f)/<(-\mathrm{id},-1)>$ where  
\[\Gamma_f(\rho_f)=\{(\gamma,u) \in \mathrm{SL}(2,\C) \times \C^{\times}:u^2=\rho_f(\gamma) \}.\]
There is an isomorphism $H^1(C_f,\C) \cong  H^0(C_f,\Omega) \oplus H^0(C_f,\Omega)^{\vee}$ and $H^0(C_f,\Omega)$ has a basis consisting of the differentials $\omega_0=dx/y$, $\omega_1=xdx/y$. 
The action of $(\gamma,u) \in \Gamma_f(\rho_f)$ on the basis $(\omega_0,\omega_1)$ equals
\[
(\gamma,u)(\omega_0,\omega_1)=(u^{-1}(c\omega_1+d\omega_0),u^{-1}(a\omega_1+b\omega_0)),
\]
see \cite[Proposition 2]{Getzler}. This tells us that if $\lambda_{\gamma}$ is an  eigenvalue of $\gamma \in \Gamma_f$ then $\xi_1=\lambda_{\gamma}u^{-1}$, $\xi_2=\lambda_{\gamma}^{-1}u^{-1}$, $\xi_3=\lambda_{\gamma}^{-1}u$ and $\xi_4=\lambda_{\gamma}u$, are the eigenvalues of $(\gamma,u) \in G_f$ acting on $H^1(C_f,\C)$. 
Finally, we need to determine the action of every $\gamma \in \Gamma_f$ on the roots of $f$, together with the point at infinity in the case that the degree of $f$ equals five. We will choose an ordering of the roots of $f$ (and possibly infinity) and denote the induced permutation by $\sigma_{\gamma}(f)$.

There are seven strata for $M_2[2]$ corresponding to the different automorphism groups $(\Gamma,\rho)$: $(C_2,\mathrm{id})$, $(C_4,\chi^2)$, $(Q_8,\chi_0)$, $(Q_{12},\chi_0)$, $(O,\chi)$, $(Q_{24},\chi_+)$ and $(C_{10},\chi^6)$. Here $C_n$ denotes the cyclic group with $n$ elements, $Q_{4n}$ the quaternionic group with $4n$ elements and $O$ is the binary octahedral group with $48$ elements. The characters are defined as in \cite[pp. 124--125]{Getzler}. The groups $\Gamma \subset \mathrm{SL}(2,\C)$ can be generated by one element $S\in \mathrm{SL}(2,\C)$ in the abelian case, and two elements $S$ and 
$U=\left(\begin{smallmatrix} 0 & 1 \\ -1 & 0 \end{smallmatrix}\right)$
in the non-abelian case. Put $\epsilon_n=e^{2\pi i/n}$. 
For further descriptions of these groups and characters, together with the computation of the Euler characteristics of the different strata, we refer to \cite[Section 2]{Getzler}. The information in the following table can be gotten from straightforward computations.

\bigskip
\vbox{
\centerline{\def\quad{\hskip 0.3em\relax}
\vbox{\offinterlineskip
\hrule
\halign{&\vrule#& \quad \hfil#\hfil \strut \quad  \cr
height2pt&\omit&&\omit&&\omit&&\omit&&\omit&&\omit&&\omit&&\omit& \cr
& $(\Gamma,\rho)$ && $f \in \Sigma(G)$ && $E_c$ && $S$ && $\rho(S)$ && $\sigma_S$ && $\rho(U)$ && $\sigma_U$ & \cr
height2pt&\omit&&\omit&&\omit&&\omit&&\omit&&\omit&&\omit&&\omit& \cr
\noalign{\hrule}
height2pt&\omit&&\omit&&\omit&&\omit&&\omit&&\omit&&\omit&&\omit& \cr
& $(C_2,\mathrm{id})$ &&   && $-1$ && $\mathrm{diag}(-1,-1)$ && $ 1 $ && $\mathrm{id} $ &&  &&  & \cr
& $(C_4,\chi^2)$ && $x^6+\alpha x^4+\beta x^2+1$ && $3$ && $\mathrm{diag}(\epsilon_4,\epsilon^{-1}_4)$ && $-1$ && $(12)(34)(56)$&&  &&  & \cr
& $(Q_8,\chi_0)$ && $x(x^4+\alpha x^2+1)$ && $-2$ && $\mathrm{diag}(\epsilon_4,\epsilon^{-1}_4)$ && $1$ && $(23)(45)$ && $-1$ && $(16)(24)(35)$& \cr
& $(Q_{12},\chi_0)$ && $x^6+\alpha x^3 -1$ && $-2$ && $\mathrm{diag}(\epsilon_6,\epsilon^{-1}_6)$ && $ 1 $ && $(123)(456)$ && $-1$ && $(14)(25)(36)$& \cr
& $(O,\chi)$ && $x(x^4+1)$ && $1$ && $\frac{-1}{\sqrt{2}}\Bigl( \begin{array}{cc} 1 & \epsilon_8 \\ \epsilon_8^3 & 1 \end{array} \Bigr)$ && $\epsilon_8^3$ && $(1264)$ && $-1$ && $(16)(23)(45)$& \cr
& $(Q_{24},\chi_+)$ && $x^6-1$ && $1$ && $\mathrm{diag}(\epsilon_{12},\epsilon^{-1}_{12})$ && $-1$ && $(123456)$ && $1$ && $(16)(25)(34)$& \cr
& $(C_{10},\chi^6)$ && $x(x^5-1)$ && $1$ && $\mathrm{diag}(\epsilon_{10},\epsilon^{-1}_{10})$ && $\epsilon_{10}^6$ && $(23456)$ &&  && & \cr
height2pt&\omit&&\omit&&\omit&&\omit&&\omit&&\omit&&\omit&&\omit& \cr
} \hrule}
}}
\bigskip
\noindent
The table above provides sufficient information to compute the contribution of $\Sigma_1(G)$ to \eqref{eq-numeric}, for all abelian groups $G$. For the non-abelian groups $G$, the information that is missing is an eigenvalue $\lambda_{\gamma}$ for all $\gamma \in G$. This problem is solved for the quaternionic groups $Q_{4n}$ by noting that it consists of the matrices $\pm S^j$ and $\pm US^j$ for $j=1,\ldots,n$, and the latter all have eigenvalues $\epsilon_4,-\epsilon_4$. Eigenvalues for the elements of the binary octahedral group can be gotten from straightforward computation. 
\subsection{Pairs of elliptic curves} 
The stratification by automorphism group for $A_1$, the moduli space of elliptic curves, is given by the three groups $C_2$, $C_4$ and $C_6$. The corresponding strata have Euler characteristics $-1$, $1$ and $1$ respectively. The two latter strata are points which can be represented by the curves $y^2=x(x^2-1)$ and $y^2=x^3-1$ respectively. The automorphism group is generated by the element $y \mapsto -y$ for $C_2$, by $y \mapsto \epsilon_4 y, x\mapsto -x$ for $C_4$ and  $y \mapsto -y, x\mapsto \epsilon_3 x$ for $C_6$. For all elliptic curves of the form $y^2=f(x)$, $H^0(C_f,\Omega)$ has a basis consisting of the differential $\omega_0=dx/y$. The eigenvalues of the induced action on $H^1(C_f,\C)$ of the generators of the automorphism groups given above then equals $-1,-1$ for $C_2$, $\epsilon_4,-\epsilon_4$ for $C_4$ and $\epsilon_6,-\epsilon_6$ for $C_6$. The induced action of the generators on the Weierstra\ss\, points (after choosing an ordering), which correspond to the roots of $f(x)$ (together with infinity) equals $\mathrm{id}$, $(12)$ and $(123)$ respectively.

Consider now $A_{1,1} \cong (A_1 \times A_1)/\fS_2$, which is the moduli space of unordered pairs of elliptic curves. This has the consequence that a pair of equal (or isomorphic) elliptic curves $E \times E$ will have an extra automorphism that sends $(p_1,p_2) \in E \times E$ to $(p_2,p_1) \in E \times E$. There will therefore be seven possible automorphism groups for $A_{1,1}$, namely $C_2 \times C_2$, $ C_2 \wr \fS_2$, $C_2 \times C_4$, $C_2 \times C_6$, $C_4 \wr \fS_2$, $C_4 \times C_6$ and $ C_6 \wr \fS_2$, where $\wr$ denotes the wreath product. The Euler characteristics of the corresponding strata are directly found to be $1,-1,-1,-1,1,1$ and $1$, respectively. Take any two elliptic curves $E_1$ and $E_2$ with automorphism groups $G_1$ and $G_2$.  Since $H^1(E_1 \times E_2,\C) \cong H^1(E_1,\C) \oplus H^1(E_2,\C)$ it is straightforward, using the information for $A_1$ above, to compute the action of $G_1 \times G_2$ if $E_1$ and $E_2$ are not isomorphic, and of $G_1 \wr \fS_2$ if $E_1$ and $E_2$ are isomorphic. The action of $G_1 \times G_2$ (and of $G_1 \wr \fS_2$) on the six Weierstra\ss\, point that are distinct from infinity on both elliptic curves is also straightforward.

\section{Isotypical decomposition in the vector-valued case} \label{sec-vv}
In this section we assume that $j>0$, so we are dealing with vector-valued Siegel modular forms.
As a consequence of \cite[Proposition 1]{BGZ}, we have
\[
M_{0,j}(\Gamma[2])=\left\{0\right\} \text{ for any } j>0.
\]
Theorem A.5 by G. Chenevier in \cite{CvdG} tells us that
\[
M_{1,j}(\Gamma[2])=S_{1,j}(\Gamma[2])=\left\{0\right\} \text{ for any } j>0.
\]

For $k=2$, there is no dimension formula for the space $M_{2,j}(\Gamma[2])$
in general. A conjectural description of the isotypical decomposition of the space $S_{2,j}(\Gamma[2])$
is given in \cite[Conjecture 1.2]{CvdG}. As this conjecture has only been verified for $j<12$ we decided to 
not implement the isotypical decomposition of the space $M_{2,j}(\Gamma[2])$ in our code.
For $k=3$, the situation is also still conjectural but with more evidence. In fact only the isotypical decomposition
of $E_{c,\mathrm{Eis}}(A_2[2],V_{j,0})$ (see Theorem \ref{thm-main} to understand how this part contributes 
to the isotypical decomposition of the space $S_{3,j}(\Gamma[2])$, and then further in Remark~\ref{rem-k3}) is still conjectural so we decided to implement 
the isotypical decomposition of the space $S_{3,j}(\Gamma[2])$ in our code. Evidence towards this conjecture are given 
for example by the results of Petersen in \cite{Petersen2017} or in Section 6.6 of \cite{BFvdG2014}. We start by recalling the
dimension formula for the space $M_{k,j}(\Gamma[2])$ which can also be used to check its conjectural isotypical decomposition for $k=3$.

\begin{thm}[{\cite[Theorems 2 and 3]{Tsushima} and \cite[Theorem 12.1]{CvdGG}}] 
For $k\geqslant 3$ odd and $j\geqslant 2$ even we have
\begin{align*}
\dim M_{k,j}(\Gamma[2])=\dim S_{k,j}(\Gamma[2])=
\frac{1}{24}
\big(&
2(j+ 1)k^3 + 3(j^2-2j-8)k^2+(j^3-9j^2-42j+118)k\\
& -2j^3-9j^2+152j-216
\big).
\end{align*}
For $k\geqslant 4$ even and $j\geqslant 2$ even  we have
\begin{align*}
\dim M_{k,j}(\Gamma[2])=
\frac{1}{24}
\big(&
2(j+ 1)k^3 + 3(j^2-2j+2)k^2+(j^3-9j^2-12j+28)k -2j^3-9j^2+182j-336
\big).
\end{align*}
\end{thm}

\begin{remark}
For $k=3$, this formula is rather pretty
\[
\dim S_{3,j}(\Gamma[2])=(j-2)(j-3)(j-4)/24.
\]
\end{remark}

\subsection{Isotypical decomposition of 
\texorpdfstring{$M_{k,j}(\Gamma[2])$}{M {k,j}(Gamma[2])}} 
For $k\geqslant 0$, the results of R\"osner (see \cite[Section 5]{Roesner}) and those of \cite[Section 13]{CvdGG} tell us that 
\[
M_{k,j}(\Gamma[2])=E_{k,j}(\Gamma[2]) \oplus S_{k,j}(\Gamma[2])=
KE_{k,j}(\Gamma[2]) \oplus S^{\textbf{(Y)}}_{k,j}(\Gamma[2]) \oplus S^{\textbf{(G)}}_{k,j}(\Gamma[2])
\]
Theorem 5.13 and Remark 5.14 in \cite{Roesner} which prove
Conjecture 6.4 of \cite{BFvdG} give us the isotypical decomposition of the space $S^{\textbf{(Y)}}_{k,j}(\Gamma[2])$ 
for $k\geqslant 3$ and $j>0$
\begin{align*}
\iso_{\fS_6} S^{\textbf{(Y)}}_{k,j}(\Gamma[2])  =
\mu_1\, s[2^3]\, + \,
\mu_2\,s[2,1^4]  + \,
\mu_3\,s[1^6] \quad \text{with} \quad
\mu_1 & = d^+_{2,j+2k-2}d^+_{2,j+2}+d^-_{2,j+2k-2}d^-_{2,j+2}\\
\mu_2 & = d_{4,j+2k-2}d_{4,j+2}\\
\mu_3 & = d^+_{2,j+2k-2}d^-_{2,j+2}+d^-_{2,j+2k-2}d^+_{2,j+2}
\end{align*}
where the integers $d_{N,k}$ and $d^{\pm}_{N,k}$ are defined as in (\ref{def_d_k}).

Proposition 13.1 in \cite{CvdGG} gives us $KE_{k,j}(\Gamma[2])={0}$ for $k$ odd.  For $k\geqslant 2$, this proposition (note that there is a typo  in \cite{CvdGG}) tells us that 
\begin{align*}
\iso_{\fS_6} KE_{2k,j}(\Gamma[2])=&
\text{Ind}_{H}^{\mathfrak{S}_6}\left( \text{{Sym}}^{j/2+k}\bigl(s[2,1]\bigr)-s[3]-s[2,1]\right)=
\text{Ind}_{H}^{\mathfrak{S}_6}\left(S_{2k+j}(\Gamma(2))\right)\\
=&d_{1,2k+j}(s[6] + s[5,1])+
(2d_{1,2k+j}+d_{2,2k+j})s[4,2] +
(d_{1,2k+j}+d_{2,2k+j})(s[3,2,1]+s[2^3])\\ 
&+  
d_{4,2k+j}(s[3,1^3]+s[2,1^4])
\end{align*}
where the last identity followed from Section \ref{EevenQ}. So to get the isotypical decomposition of 
the space $M_{k,j}(\Gamma[2])$ it remains to determine it for the space $S^{\textbf{(G)}}_{k,j}(\Gamma[2])$. This is done in the next section.

\begin{subsection}{An isotypical dimension formula for 
\texorpdfstring{$S^{\textbf{(G)}}_{k,j}(\Gamma[2])$}{S {k,j}(Gamma[2])}}
First we introduce some notation from \cite{BFvdG}. Let 
\[
\begin{array}{ll}
A=s[6]\oplus s[5,1]+s[4,2],     &\qquad A'=s[6]\oplus s[4,2]\oplus s[2^3], \\
B=s[4,2]\oplus s[3,2,1]+s[2^3], &\qquad B'=s[5,1]\oplus s[4,2]\oplus s[3,2,1], \\
C=s[3,1^3]\oplus s[2,1^4],      &\qquad C'=s[4,1^2]\oplus s[3^2].
\end{array}
\]

For any $l,m$, with $l > m > 0$, put $n=l+m+4$, $n'=l-m+2$ and define
\begin{align}\label{EisCohom}
\nonumber
E_{c,\mathrm{Eis}}(A_2[2],V_{l,m})=&
(d_{1,n'}-d_{1,n})\,(A'+B')+\,(d_{2,n'}-d_{2,n})\,B'+(d_{4,n'}-d_{4,n})\, C' 
+\frac{1}{2}\bigl(1+(-1)^m\bigr)\, (A+B)\\
&+2\bigl((d_{1,m+2}-d_{1,l+3})\,(A+B) +(d_{2,m+2}-d_{2,l+3})\,B+ 
(d_{4,m+2}-d_{4,l+3})\,C\bigr)
\end{align}
and 
\begin{align*}
 E_{c,\mathrm{endo}}(A_2[2],V_{l,m})= & 
 -2\Bigl(
 d_{4,n'} \, \bigl(d_{4,n} \, s[3,1^3]+d_{1,n} \, s[3^2]+(d_
{1,n}+d_{2,n}) \, s[4,1^2] \bigr)\\
& + d_{2,n'} \,  \bigl((d_{1,n}+d_{2,n}) \, s[3,2,1]+d_{4,n} \, s[4,1^2] + d_{1,n} \, s[4,2]+d_{1,n} \, s[5,1] \bigr)  \\
& + d^+_{2,n'}\, \bigl(d^+_{2,n} \, s[4,2]+d^-_{2,n}\, s[5,1]\bigr)
+d^-_{2,n'}\, \bigl(d^-_{2,n} \, s[4,2]+d^+_{2,n} \, s[5,1]\bigr) \\
& + d_{1,n'} \, \bigl(d_{1,n}\,(A'+B') +d_{2,n}\,B'+ d_{4,n}\,C'\bigr) 
\Bigr)
\end{align*}
as elements of the representation ring $\Z[\fS_6]$.

\begin{thm} \label{thm-main} 
For any $k\geq 4$ and $j >0$, put $l=j+k-3$ and $m=k-3$. Then 
\[
\iso_{\fS_6} S^{\textbf{(G)}}_{k,j}(\Gamma[2])=-\frac{1}{4}\Bigl(E_c(A_2[2],V_{l,m})-E_{c,\mathrm{Eis}}(A_2[2],V_{l,m})-E_{c,\mathrm{endo}}(A_2[2],V_{l,m})+2 \, \iso_{\fS_6}  S^{\textbf{(Y)}}_{k,j}(\Gamma[2])\Bigr).
\]
\end{thm}

\begin{proof}
In \cite{BFvdG}, the compactly supported $\ell$-adic Euler characteristic of local systems $\V_{l,m}$ 
on $\mathcal A_2[2]$ taking values in the Grothendieck group of (absolute) Galois representations is 
decomposed into the following pieces, 
\[
e_c(\mathcal A_2[2],\V_{l,m})=e_{c,\mathrm{Eis}}(\mathcal A_2[2],\V_{l,m})+e_{c,\mathrm{endo}}(\mathcal A_2[2],\V_{l,m})- S[l-m,m+3,\Gamma[2]]. 
\] 
The formula for $E_{c,\mathrm{Eis}}(A_2[2],V_{l,m})$ (respectively $E_{c,\mathrm{endo}}(A_2[2],V_{l,m})$) 
is found by taking dimensions in the formula for $e_{c,\mathrm{Eis}}(\mathcal A_2[2],\V_{l,m})$ 
(respectively $e_{c,\mathrm{endo}}(\mathcal A_2[2],\V_{l,m})$) in \cite[Theorem 4.4]{BFvdG} 
(respectively \cite[Conjecture 7.1]{BFvdG}). The representation $S[l-m,m+3,\Gamma[2]]$ should 
conjecturally consist of $2$-dimensional pieces for each Hecke eigenvector in $S^{\textbf{(Y)}}_{k,j}(\Gamma[2])$, 
with isotypic decomposition given in \cite[Conjecture 6.4]{BFvdG}, and $4$-dimensional pieces 
for each Hecke eigenvector in $S^{\textbf{(G)}}_{k,j}(\Gamma[2])$. 

The conjectural description in \cite{BFvdG} described above, has been proven in \cite{Roesner}. Conjectures 7.1 and 6.4 of \cite{BFvdG} are proved by Theorem 5.13 of \cite{Roesner}, see Remark 5.14 of \cite{Roesner}. The result then follows from \cite[Corollary 5.20]{Roesner}.
\end{proof}

\begin{remark} \label{rem-k3}
In \cite{BFvdG}, directly after Theorem~4.4, it is conjectured that $E_{c,\mathrm{Eis}}(A_2[2],V_{l,0})$ 
for any $l>0$ is given by (\ref{EisCohom}), with the difference that one needs to put $d_{1,2}=-1$. 
If we assume this conjecture to be true, and we define $E_{c,\mathrm{endo}}(A_2[2],V_{l,0})$ 
for any $l>0$ using the formula above, then Theorem~\ref{thm-main} also holds for $k=3$ and $j>0$ 
using the same proof (and the same results of \cite{Roesner}). 
\end{remark}

\end{subsection}

\section*{Acknowledgement}
The second author was supported by the Simons Foundation Award 546235 at the Institute for Computational and Experimental Research in Mathematics at Brown University. We thank Eran Assaf and Gerard van der Geer for useful discussions.


\begin{thebibliography}{9999}

\bibitem{Arakawa} T. Arakawa: 
{\sl Vector valued Siegel  modular forms of degree two 
and the associated Andrianov L-functions.} 
Manuscripta Math. 44(1983), 155--185.

\bibitem{Arthur} J. Arthur:
{\sl Automorphic representations of GSp(4).} 
Contributions to automorphic forms, geometry, and number theory, 65--81, Johns Hopkins Univ. Press, Baltimore, MD, 2004.

\bibitem{BC} 
J. Bergstr\"om, F. Cl\'ery:
{\sl Siegel Modular Forms of degree 2  and Level 2 - Dimensions.}
\url{https://github.com/CleryFabien/SMF_Degree_2_Level_2_Dim}

\bibitem{BFvdG} 
J. Bergstr\"om, C. Faber, G. van der Geer:
{\sl Siegel modular forms of genus 2 and level 2: cohomological computations and conjectures.}
Int. Math. Res. Not. IMRN 2008, Art. ID rnn~100.

\bibitem{BFvdG2014} 
J. Bergstr\"om, C. Faber, G. van der Geer:
{\sl Siegel modular forms of degree three and the cohomology of local systems.}
Selecta Math. (N.S.) \textbf{20} (2014), no.1, 83--124.

\bibitem{BCFvdGweb} 
J. Bergstr\"om, F. Cl\'ery, C. Faber, G. van der Geer:
{\sl  Siegel Modular Forms of Degree Two and Three.}
\url{http://smf.compositio.nl/}

\bibitem{BvdG} 
J. Bergstr\"om, G. van der Geer:
{\sl Picard modular forms and the cohomology of local systems on a Picard modular surface.}
Comment. Math. Helv. \textbf{97} (2022), no. 2, 305--381. 

\bibitem{BGZ}
J.H. Bruinier, G. van der Geer, G. Harder, D. Zagier:
{\sl The 1-2-3 of modular forms.}
Lectures from the Summer School on Modular Forms and their Applications held in Nordfjordeid, June 2004. 
Universitext
Springer-Verlag, Berlin, 2008. x+266 pp.

\bibitem{Bolza} O. Bolza:
{\sl On binary sextics with linear transformations into themselves.}
Amer. J. Math. \textbf{10} (1887), no. 1, 47--70. 

\bibitem{CFvdG19} F. Cl\'ery, C. Faber, G. van der Geer:
{\sl Covariants of binary sextics and modular forms of degree 2 with character.} 
Math. Comp. \textbf{88} (2019), no. 319, 2423--2441.

\bibitem{CvdG}  F. Cl\'ery, G. van der Geer:
{\sl On vector-valued Siegel modular forms of degree 2 and weight $(j,2)$.} 
With two appendices by Ga\"etan Chenevier. Doc. Math. \textbf{23} (2018), 1129--1156.

\bibitem{CvdGG} F. Cl\'ery, G. van der Geer, S. Grushevsky: 
{\sl Siegel modular forms of genus 2 and level 2.} 
Internat. J. Math. \textbf{26} (2015), no. 5, 1550034, 51 pp.

\bibitem{Freitag79} E. Freitag:
{\sl Ein Verschwindungssatz f\"ur automorphe Formen zur Siegelschen Modulgruppe.}
Math. Z. 165 (1979), no. 1, 11--18.

\bibitem{FH}W.\ Fulton, J.\ Harris: {\sl Representation Theory.
A First Course.} Springer-Verlag, New York, 1991.

\bibitem{Getzler} E. Getzler:
{\sl Euler characteristics of local systems on $\mathcal{M}_2$.} 
Compositio Math. \textbf{132} (2002), no. 2, 121--135. 

\bibitem{Ibukiyama84} T. Ibukiyama:
{\sl On symplectic Euler factors of genus two.}
J. Fac. Sci. Univ. Tokyo Sect. IA Math., \textbf{30} (1984), no. 3, 587--614.

\bibitem{Ibukiyama07} T. Ibukiyama: 
{\sl Dimension formulas of Siegel modular forms of weight 3 and supersingular
abelian surfaces.} In Proceedings of the 4-th Spring Conference. Abelian Varieties and Siegel Modular Forms, 39--60, 2007.

\bibitem{Igusa64}  J.-I. Igusa:
{\sl  On Siegel modular forms genus two. II.} Amer. J. Math. \textbf{86} (1964), 392--412. 

\bibitem{MartinK} K. Martin: 
{\sl Refined dimensions of cusp forms, and equidistribution and bias of signs.}
J. of Number Theory \textbf{188} (2018) 1--17.

\bibitem{Petersen} D. Petersen:
{\sl Cohomology of local systems on loci of d-elliptic abelian surfaces.}
Michigan Math. J. \textbf{62} (2013), no.4, 705--720.

\bibitem{Petersen2017} D. Petersen:
{\sl Cohomology of local systems on the moduli of principally polarized abelian surfaces.}
Pacific J. Math. \textbf{275} (2015), no.1, 39--61.

\bibitem{Roesner} M. R\"osner: 
{\sl Parahoric Restriction for GSp(4) and the Inner Cohomology of Siegel Modular Threefolds.}
Thesis-Ruprecht-Karls-Universit\"at Heidelberg (Germany), DOI: 10.11588/heidok.00021401
\url{https://archiv.ub.uni-heidelberg.de/volltextserver/21401/1/Dissertation_Roesner.pdf}

\bibitem{Roy_Schmidt_Yi} M. Roy, R. Schmidt, S. Yi:
{\sl Dimension formulas for Siegel modular forms of level 4.}
Mathematika \textbf{69} (2023), no. 3, 795--840.

\bibitem{Schmidt18} R. Schmidt:
{\sl  Packet structure and paramodular forms.}
Trans. Amer. Math. Soc. 370 (2018), no. 5, 3085--3112. 

\bibitem{Schmidt} R. Schmidt:
{\sl  Dimension formulas for spaces of Siegel modular forms of degree 2.}
\url{https://math.ou.edu/~rschmidt/dimension_formulas/}

\bibitem{Stein} W. Stein:
{\sl Modular forms, a computational approach.}
With an appendix by Paul E. Gunnells
Grad. Stud. Math., \textbf{79}
American Mathematical Society, Providence, RI, 2007. xvi+268 pp.

\bibitem{Tsushima1982} R. Tsushima:
{\sl  On the spaces of Siegel cusp forms of degree two.} 
Amer. J. Math. \textbf{104} (1982), no. 4, 843--885. 

\bibitem{Tsushima} R. Tsushima:
{\sl An explicit dimension formula for the spaces of generalized automorphic forms with respect to $\Sp(2,\Z)$.}
Proc. Japan Acad. Ser. A Math. Sci. \textbf{59} (1983), no. 4, 139--142.

\bibitem{Wakatsuki} S. Wakatsuki:
{\sl  Dimension formulas for spaces of vector-valued Siegel cusp forms of degree two.} 
J. Number Theory \textbf{132} (2012), no. 1, 200--253. 

\end{thebibliography}
\end{document}